\documentclass[11pt]{article}
\usepackage{amsmath,amsthm,amssymb,ifsym,a4wide} 
\usepackage[twoside]{geometry}
\usepackage{graphicx}
\usepackage{epsfig} 
\usepackage{slashed}  
\usepackage[hidelinks]{hyperref}
\usepackage{cite}

\usepackage{appendix}

\usepackage{ulem}

\newtheorem{theorem}{Theorem}[section]
\newtheorem{proposition}[theorem]{Proposition}
\newtheorem{lemma}[theorem]{Lemma}

\newtheorem{definition}[theorem]{Definition}
\newtheorem{remark}[theorem]{Remark} 
\numberwithin{equation}{section} 
\numberwithin{figure}{section}  
\usepackage{amssymb, amscd, mathrsfs, wasysym}

\newcommand{\myfootnote}[1]{
    \renewcommand{\thefootnote}{}
    \footnotetext{\hspace{-16.5pt}\scriptsize#1}
    \renewcommand{\thefootnote}{\arabic{footnote}}
}

\newcommand \la \langle
\newcommand \ra \rangle

\newcommand \Ecal {\mathcal{E}}

\newcommand \trianglerightNEW \triangleright

\newcommand \auth {\textsc} 

\newcommand \CC {\mathbb C}

\newcommand \bei {\begin{itemize}}
\newcommand \eei {\end{itemize}}
\newcommand \be {\begin{equation}}
\newcommand \bel {\be\label}
\newcommand \ee {\end{equation}}
\newcommand \del \partial
\newcommand \RR {\mathbb R}

\newcommand \eps \epsilon 

\newcommand \Scal{\mathcal{S}}

\usepackage[most]{tcolorbox} 

\let\oldmarginpar\marginpar
\renewcommand\marginpar[1]{\-\oldmarginpar[\raggedleft\footnotesize #1]%
{\raggedright\footnotesize #1}}
\begin{document}

\title{\bf \Large 
Global Existence and Scattering of the Klein-Gordon-Zakharov System in Two Space Dimensions}

\author{Shijie Dong${}^{\,\text{1}}$, 
Yue Ma${}^{\,\text{2}, \ast}$,
}


\date{\today}
\maketitle

\begin{abstract} 
{
We are interested in the Klein-Gordon-Zakharov system in $\mathbb{R}^{1+2}$, which is an important model in plasma physics with extensive mathematical studies. The system can be regarded as semilinear coupled wave and Klein-Gordon equations with nonlinearities violating the null conditions. Without the compactness assumptions on the initial data, we aim to establish the existence of small global solutions, and in addition, we want to illustrate the optimal pointwise decay of the solutions. Furthermore, we show that the Klein-Gordon part of the system enjoys linear scattering while the wave part has uniformly bounded low-order energy.
None of these goals is easy because of the slow pointwise decay nature of the linear wave and Klein-Gordon components in $\mathbb{R}^{1+2}$.
We tackle the difficulties by carefully exploiting the properties of the wave and the Klein-Gordon components, and by relying on the ghost weight energy estimates to close higher-order energy estimates. This appears to be the first pointwise decay result and the first scattering result for the Klein-Gordon-Zakharov system in $\RR^{1+2}$ without compactness assumptions.
}
\end{abstract}

\tableofcontents

\myfootnote{
${}^\ast$Corresponding author.\\
${}^\text{1}$ Fudan University, School of Mathematical Sciences, 220 Handan Road, Shanghai, 200433, P.R. China. 
Email: shijiedong1991@hotmail.com\\
${}^\text{2}$ Xi’an Jiaotong University, School of Mathematics and Statistics, 28 West Xianning Road, Xi’an, Shaanxi 710049, P.R. China. 
Email: yuemath@xjtu.edu.cn.
\\
{\sl AMS :} 35L05. 
}

\section{Introduction}

\paragraph{Model problem and main results}

We consider the Klein-Gordon-Zakharov model in $\RR^{1+2}$, which is an important model in plasma physics with extensive mathematical studies. The model equations are as follows
\bel{eq:model-KGZ}
\aligned
&-\Box E + E = - n E,
\\
&-\Box n = \Delta |E|^2.
\endaligned
\ee
The unknowns include the electronic field $E = (E^1, E^2)$ taking values in\footnote{Originally $E$ takes values in $\RR^2$, but more general cases of taking values in $\CC^{N_0}$ with $N_0 = 1, 2, \cdots$ can also be treated.} $\RR^2$, and the ion density $n$ taking values in $\RR$. The Klein-Gordon-Zakharov equations can be regarded as a semilinear coupled wave and Klein-Gordon system, with Klein-Gordon field $E$ and wave field $n$.
In the spacetime $\RR^{1+2}$ we adopt the signature $(-, +, +)$. The wave operator is denoted by $\Box = \del_\alpha \del^\alpha$, and $\Delta = \del_a \del^a$ represents the Laplace operator. Throughout Greek letters $\alpha, \beta, \cdots \in \{0, 1, 2\}$ denote spacetime indices, while Latin letters $a, b, \cdots \in \{1, 2\}$ are used to represent space indices. 
The Einstein summation convention is adopted unless otherwise specified.


We consider the Cauchy problem associated to \eqref{eq:model-KGZ} with initial data on the slice $t=t_0 =0$
\bel{eq:model-ID}
\aligned
\big( E, \del_t E \big) (t_0, \cdot)
=
\big( E_0, E_1 \big),
\qquad
\big( n, \del_t n \big) (t_0, \cdot)
=
\big( n_0, n_1 \big) := \big( \Delta n^{\Delta}_0, \Delta n^{\Delta}_1 \big), 
\endaligned
\ee
and the functions $(E_0, E_1, n^\Delta_0, n^\Delta_1)$ are assumed to be sufficiently smooth, but they {\sl do not} need to be compactly supported. The main objective of the present article is the following asymptotic stability result associated to small regular initial data together with the scattering property on the Klein-Gordon components, i.e., the Langmuir wave (stated in the next Theorem).

\begin{theorem}\label{thm:main1}
Consider the Klein-Gordon-Zakharov system in \eqref{eq:model-KGZ}, and let $N \geq 14$ be a large integer. There exits $\eps_0 >0$, such that for all initial data satisfying the smallness condition 
\be 
\aligned
&\sum_{0\leq j\leq N+2}\| \la |x| \ra^{j+1} \log (1+\la |x| \ra) E_0 \|_{H^j} + \sum_{0\leq j\leq N+1} \| \la |x| \ra^{j+2} \log (1+\la |x| \ra) E_1 \|_{H^j}
\\
+
& \sum_{0\leq j\leq N+3} \| \la |x| \ra^{j+1}  n^\Delta_0 \|_{H^j} + \sum_{0\leq j\leq N+2} \| \la |x| \ra^{j+2} n^\Delta_1 \|_{H^j}
\leq \eps < \eps_0,
\endaligned
\ee
the initial value problem \eqref{eq:model-KGZ}--\eqref{eq:model-ID} admits a global solution $(E, n)$, 
which enjoys the following optimal pointwise decay results
\bel{eq:thm-decay}
|E(t, x)| \lesssim t^{-1},
\quad
|n(t, x) | \lesssim t^{-1/2} \langle t-r \rangle^{-1/2}.
\ee

\end{theorem}

The global existence for the Klein-Gordon-Zakharov system (with some first order equations) in two space dimensios was proved in \cite{Guo}, but whether the system is stable is unknown. Our result in Theorem \ref{thm:main1} verifies that the system is not only stable but also asymptotically stable. 

A similar version of Theorem \ref{thm:main1} was demonstrated in \cite{Dong2006, Ma2008} with compactly supported initial data. Now we can treat the non-compactly supported initial with decay at infinity.
At this point, we recall the global existence result \cite{Stingo18} regarding a quasilinear wave-Klein-Gordon model satisfying the null condition in two space dimensions, where the weights are lower than our result in Theorem \ref{thm:main1}.

Our next result states that the Klein-Gordon field $E$ scatters linearly.

\begin{theorem}\label{thm:scatter}
Let the same assumptions in Theorem \ref{thm:main1} hold, then there exists a pair of functions
$$
(E^+_0, E^+_1) \in H^{N-7} \times H^{N-8},
$$
such that 
\bel{eq:thm-scatter}
\|(E-E^+)(t, \cdot) \|_{H^{N-7}} + \| \del_t (E-E^+)(t, \cdot) \|_{H^{N-8}}
\leq C \la t \ra^{-1/4}
\to 0,
\qquad
\text{as } t \to +\infty,
\ee
in which $E^+$ is a linear Klein-Gordon component solving
$$
\aligned
&-\Box E^+ + E^+ = 0,
\\
&\big( E^+, \del_t E^+ \big)(t_0) = (E^+_0, E^+_1).
\endaligned
$$

\end{theorem}

We want to emphasize that the scattering result in Theorem \ref{thm:scatter} is valid under quite high regularity assumptions on the initial data. As a comparison, we recall that the scattering result of the Zakharov equations in $\RR^{1+3}$ in \cite{Hani} is also obtained with high regularity assumptions on the initial data. Our scattering result is different from the one proved with (radial) initial data in low regularity for model \eqref{eq:model-KGZ} in \cite{Guo-N-W} in $\RR^{1+3}$, where very different difficulties arise. See in detail below.

We note that our method cannot assert whether the wave part $n$ scatters linearly or not in the energy space (i.e., $\sum_{a=1,2}\|\del_a n\| + \|\del_t n\|$), and we leave it open. However, we will show that the energy of the wave component $n$ is uniformly bounded in time (see \eqref{eq:X-norm}), which is necessary to linear scattering.

There exist already several global existence results for two dimensional coupled wave and Klein-Gordon equations with different types of nonlinearities, but most of the results were shown under the assumption that the initial data are compactly supported; see \cite{Ma19, DW2105} and the references therein for such cases. The ideas and techniques used in proving Theorems \ref{thm:main1}--\ref{thm:scatter} are expected to have further applications, such as to remove the compactness assumptions on the existing results, or to study coupled wave and Klein-Gordon systems with more general nonlinearities of physical or mathematical interests.

\paragraph{Background and historical notes}
The Klein-Gordon-Zakharov system was originally introduced in \cite{{Zakharov}}, which describes the interaction between Langmuir waves and ion sound waves in plasma; see \cite{Dendy} for more of its physical background. The global existence as well as the pointwise decay result on this system in $\RR^{1+3}$ was established  dating back to \cite{OTT}, and then in many other context (see for instance the recent work \cite{Dong2101}). However, due to the insufficiency of the decay in lower dimension (see in detail below), the global existence problem in $\RR^{1+2}$ is somewhat more challenging. In \cite{Dong2006} a global existence result, with pointwise asymptotics of the solution, is established on localized restricted initial data and then it is generalized in \cite{Duan-Ma}. These results are established within the so-called hyperboloidal foliation framework and thus demand that the initial data being compactly supported. In the present work we rely on a global iteration framework, which was used for instance in \cite{Dong2006}, to remove this restriction.

We next recall some mathematical studies in plasma physics which are relevant to our results. Being a highly important model, the Klein-Gordon-Zakharov system can be derived (with certain assumptions) from Euler-Maxwell system, which can be found in \cite[Section 2.1]{Colin}.
The Euler-Maxwell model is one of the most fundamental models in plasma physics, which describes laser-plasma interactions.
In the seminal work of Guo-Ionescu-Pausader \cite{Guo-I-P}, the two-fluid Euler-Maxwell model was shown to admit smooth solutions in $\RR^{1+3}$. We recall that the Euler-Poisson system in $\RR^{1+2}$ was proved to have global solutions by Li-Wu \cite{Li-Wu14} and Ionescu-Pausader \cite{Ionescu-P4}, and later on, the one-fluid Euler-Maxwell system in $\RR^{1+2}$ was proved to have global solutions by Deng-Ionescu \cite{Deng-I-P}, and both the systems can be reduced to Klein-Gordon equations. 

Back to the Klein-Gordon-Zakharov context,
besides the global existence and pointwise asymptotics results, there is also plenty of work concerning other problems around this system. In $\RR^{1+3}$, Ozawa, Tsutaya, and Tsutsumi \cite{OTT2} showed the Klein-Gordon-Zakharov equations admit global solutions for low regularity initial data under the condition that the propagation speeds are different in two equations. In a series papers \cite{Masmoudi0, Masmoudi2}, Masmoudi and Nakanishi investigated the limiting system as certain parameters go to $+\infty$ in the Klein-Gordon-Zakharov equations. We also recall the relevant work by Shi and Wang \cite{WangShu}, where a finite time blow-up result was obtained for low regular initial data satisfying certain conditions. Again in $\RR^{1+3}$, the linear scattering of the Klein-Gordon-Zakharov equations in radial case (with low regular initial data) was obtained in the works \cite{Guo-N-W, Guo-N-W2} by Guo-Nakanishi-Wang, while the linear scattering of the Zakharov equations (with high regular initial data) was illustrated by Hani-Pasateri-Shatah \cite{Hani}.

The coupled wave and Klein-Gordon equations in two space dimensions have received much attention, and substantial progress has been made regarding its small data global existence problem in recent years. We are not going be exhaustive here, but instead leading one to the works \cite{Stingo18, Ma19, DW2105} and the references therein for more discussions.

\paragraph{Major difficulties and technical contributions}
Concerning the Klein-Gordon-Zakharov system \eqref{eq:model-KGZ}, there are many difficulties in obtaining the results in Theorems \ref{thm:main1}--\ref{thm:scatter}. Besides the lack of the scaling vector field in the analysis, we demonstrate some key issues encountered in the study of the system \eqref{eq:model-KGZ}. The first one is the absence of null structure in the nonlinearities of the system \eqref{eq:model-KGZ}. We remark that the right-hand-side of the system violates the classical null condition of Christodoulou-Klainerman. The second comes form the low decay rate of both wave and Klein-Gordon equations in lower dimension $\RR^{1+2}$ and the third one is about dealing with the non-compactly supported initial data. Let us explain in detail these obstacles and our strategies aimed at each of them. 

In the research of nonlinear wave systems (including wave-Klein-Gordon systems), the null condition (in the sense of Christodoulou-Klainerman) plays an essential role. Roughly speaking, it provides additional decay near the light cone (i.e., the region $t$ close to $|x|$), where the wave equations fail to have sufficiently fast decay. 
We note the only existing global existence results on two dimensional coupled wave and Klein-Gordon equations with non-cpmpactly supported initial data are due to \cite{Dong2005, Stingo18}, where all of the nonlinearities are assumed to obey the null condition, and thus the situation we consider here is more difficult.
To conquer this difficulty, we observe and take full use of a special structure of the system \eqref{eq:model-KGZ}: in the right-hand-side of the wave equation the only quadratic term is a wave-Klein-Gordon mixed one. The Klein-Gordon components enjoy an additional decay rate expressed as $\la t+r\ra^{-1}\la r-t\ra$ near the light cone (see Proposition \ref{prop:KG-extra} for more details), which will compensate the absence of the null structure in our analysis. 

The insufficiency of decay in the lower dimension brings another difficulty. In $\RR^{1+2}$, the free-linear waves decay at the speed of $t^{-1/2}$, while free-linear Klein-Gordon components decay at the speed of $t^{-1}$. This means that the best we can expect for the nonlinearities is
\bel{eq:decay-problem}
\big\| n E \big\|_{L^2(\RR^2)} 
\lesssim t^{-1},
\qquad
\big\| \Delta |E|^2 \big\|_{L^2(\RR^2)} 
\lesssim t^{-1},
\ee
which are non-integrable with respect to time. Thus under this situation, it is highly non-trivial to prove the sharp pointwise decay results,  as well as closing the bootstrap, of $E$ and $n$.
Our strategy of solving this thorny issue of the slow decay in the $nE$ nonlinearity follows. We first reveal a Hessian structure $\Delta n^\Delta E$ with the relation $n=\Delta n^\Delta$ relying on the special structure in the wave equation of $n$. Then we employ different techniques in different spacetime regions to obtain an extra $\la t-r \ra$ decay of the Hessian form of the wave components; see Propositions \ref{prop:wave-extra} and \ref{prop:wave-extra2}. Finally, combined with the extra $\la t+r\ra^{-1}\la r-t\ra$ decay of the Klein-Gordon component $E$, the extra $\la t-r \ra$ decay of $\Delta n^\Delta$ is transformed into $\la t+r \ra$ near the light cone, which is favorable.

The third difficulty is the most severe one, and it is the main interest of the present article to tackle. When the initial data is compactly supported, the system \eqref{eq:model-KGZ} has been discussed within the hyperboloidal foliation framework, see for example \cite{Dong2006, Ma2008, Duan-Ma}. However, there due to the essential short board of the hyperboloidal foliation, one can not analyze the solution outside of the light cone, and thus, the demand on the compactness of the support of initial data became inevitable. Here we apply another strategy used in \cite{Dong2005, Dong2101} which is entirely different from the hyperboloidal foliation. Equipped with these techniques, one manages to treat the whole spacetime in its entirety. This allows us to remove the restriction on the initial data.
As a comparison, in the proofs of \cite{Dong2005, Dong2101} we close the iteration by relying on a global decay of the Klein-Gordon components. But here we need to investigate more detailed properties of the Klein-Gordon field $E$ in different spacetime regions (see for instance \eqref{eq:X-norm}) so that we can manage to close the proof.

Once the global solution is established for the system \eqref{eq:model-KGZ}, a natural question arises: will the global solution $(E, n)$ scatters to the linear case? This is the second objective of the present work. To show the linear scattering of the Klein-Gordon-Zakharov system \eqref{eq:model-KGZ} (or the Zakharov equations) is a tough problem even in $\RR^{1+3}$. In \cite{Guo-N-W, Guo-N-W2}, the Klein-Gordon-Zakharov system was shown to enjoy the linear scattering for (radial) initial data with low regularity, and later on in \cite{Hani}, the Zakharov system was proved to scatter linearly for initial data with high regularity. All of the works \cite{Guo-N-W, Guo-N-W2, Hani} are proved in $\RR^{1+3}$, and the proofs cannot be directly applied to the two dimensional cases. Based on a scattering result on wave and Dirac equations in \cite{Katayama17, Dong-Li21}, we succeed in showing that the Klein-Gordon part $E$ in the system \eqref{eq:model-KGZ} enjoys linear scattering in its energy space (as illustrated in Theorem \ref{thm:scatter}) by adapting the result in \cite{Katayama17, Dong-Li21} to Klein-Gordon equations. We cannot prove whether the wave part $n$ scatters linearly or not, but instead we show that the natural wave energy of $n$ (i.e., $\sum_{a=1,2} \|\del_a n\| + \|\del_t n\|$) is uniformly bounded in time, which is a weaker result (also a necessary result to linear scattering). As an interesting comparison, we refer to \cite{Ionescu-P2} the case of wave-Klein-Gordon system in $\RR^{1+3}$, where neither the wave components nor the Klein-Gordon one scatters linearly.



\paragraph{Further discussions}

Many fundamental physical models are governed by the coupled wave and Klein-Gordon equations, including the Einstein-Klein-Gordon equations, the Dirac-Klein-Gordon equations, the Klein-Gordon-Zakharov equations studied in the present paper, the Maxwell-Klein-Gordon equations etc, and the asymptotic behavior of the equations is relatively well-studied in three space dimensions. But there are few results for these systems of equations in two space dimensions, and the pointwise asymptotics of the solutions (even for small smooth initial data with compact support) are also unknown except the Klein-Gordon-Zakharov equations \cite{Dong2006, Ma2008, Duan-Ma} and some cases of the Dirac-Klein-Gordon equations \cite{DW2105}. We believe it is of great mathematical and physical significance to show the asymptotic behavior (or some other related results) of such equations in two space dimensions.

We recall that the Klein-Gordon-Zakharov equations were shown to have uniformly bounded (low- and high-order) energy in the recent work \cite{Dong2101} in $\RR^{1+3}$, where the pointwise decay of the solutions is faster compared to the two dimensional case. However in $\RR^{1+2}$, due to the critical decay rates of the nonlinearities as illustrated in \eqref{eq:decay-problem}, we do not expect the high-order energy to be uniformly bounded in time unless some new observations on the system \eqref{eq:model-KGZ} are found.

In the present paper, our focus is to handle non-compactly supported initial data for the Klein-Gordon-Zakharov equations \eqref{eq:model-KGZ} in two space dimensions. We lead one to \cite{PLF-YM-arXiv1, Ionescu-P2, Klainerman-Wang-Yang} for other methods of treating non-compactly supported initial data, which have the potential to lead to new results.



\subsection*{Outline}

The rest of this article is organised as follows.

In Section \ref{sec:pre}, we introduce the preliminaries and some fundamental energy estimates for wave and Klein-Gordon equations. We then explore some extra decay properties for the wave and the Klein-Gordon components in Section \ref{sec:extra-decay}. In Section \ref{sec:global-existence}, we demonstrate the proof of the global existence and the pointwise decay results for the Klein-Gordon-Zakharov equations in Theorem \ref{thm:main1} relying on the contraction mapping theorem. Last, we show the scattering result for the Klein-Gordon component in Theorem \ref{thm:scatter} in Section \ref{sec:scatter} with some supporting materials in Appendix \ref{sec:appendix}.


\section{Preliminaries}
\label{sec:pre}


\subsection{Basic notation}
 
We work in the $(2+1)$ dimensional spacetime with signature $(-, +, +)$, i.e., the Minkowski metric $\eta = \text{diag}\{-1, 1, 1\}$. A point in $\RR^{1+2}$ is denoted by $(x_0, x_1, x_2) = (t, x_1, x_2)$, and its spacial radius is written as $r = \sqrt{x_1^2 + x_2^2}$. We use Latin letters to represent space indices $a, b, \cdots \in \{1, 2\}$, while Greek letters are used to denote spacetime indices $\{0, 1, 2\}$, and the indices are raised or lowered by the metric $\eta$.

We first recall the vector fields which will be frequently used in the analysis.
\bei
\item Translations: $\del_\alpha$, 

\item Rotations: $\Omega_{ab} = x_a \del_b - x_b \del_a$,

\item Lorentz boosts: $L_a = x_a \del_t + t \del_a$, 

\item Scaling vector field: $L_0 = S = t \del_t + r \del_r$.

\eei
Excluding the scaling vector field $L_0$, we utilize $\Gamma$ to denote a general vector field in the set
$$
V := \{ \del_\alpha, \Omega_{ab}, L_a \}.
$$
Besides, the following good derivatives
$$
G_a 
:= r^{-1} \big(x_a \del_t + r \del_a \big),
$$ 
will appear in Alinhac's ghost weight method.

We define (and fix) a smooth  cut-off function which is increasing and satisfies
\begin{eqnarray}\label{eq:cut-off}
\chi(s):=
\left\{
\begin{array}{lll}
0, & \quad s\leq 1,
\\
1, &\quad s \geq 2.
\end{array}
\right.
\end{eqnarray} 
This will be frequently used to derive energy estimates in different spacetime regions.


\subsection{Energy estimates}
We consider the wave-Klein-Gordon equation with $m=0, 1$
\bel{eq:w-KG-m}
\aligned
-\Box u + m^2 u = F_u.
\endaligned
\ee
We will demonstrate several types of energy estimates for the equation \eqref{eq:w-KG-m}. We recall the energy functional (with $\delta>0$)
\bel{eq:gst-functional}
\Ecal_{gst, m} (t, u)
:=
\Ecal_m (t, u) 
+ 
\sum_a \int_{t_0}^t \int_{\RR^2} {\delta|G_a u|^2 \over \langle \tau-r\rangle^{1+\delta}} \, dxd\tau
+
 \int_{t_0}^t \int_{\RR^2} {\delta m^2| u|^2 \over \langle \tau-r\rangle^{1+\delta}} \, dxd\tau,
\ee
in which the natural energy is defined by
$$
\Ecal_m (t, u) 
:=
\int_{\RR^2} \Big( |\del_t u|^2 + \sum_a |\del_a u|^2 + m^2 |u|^2\Big)(t,\cdot) \, dx.
$$
The abbreviations $\Ecal (t, u) = \Ecal_0 (t, u) $ and $\Ecal_{gst} (t, u) = \Ecal_{gst, 0} (t, u) $ will be used. We also remark that when $t=t_0$,  $\Ecal_{gst,m}(t_0,u) = \Ecal_m(t_0,u)$.

The natural energy estimates for wave-Klein-Gordon equations read.
\begin{proposition}
Consider \eqref{eq:w-KG-m} with $m=0, 1$, and it holds both
\bel{eq:natural-wKG1}
\aligned
\Ecal_{ m} (t, u)
\lesssim
\Ecal_{ m} (t_0, u)
+
\int_{t_0}^t \int_{\RR^2} \big| F_u \big| \, \big| \del_t u \big| \, dxd\tau
\endaligned
\ee
and
\bel{eq:natural-wKG2}
\aligned
\Ecal_{m} (t, u)^{1/2}
\lesssim
\Ecal_{m} (t_0, u)^{1/2}
+
\int_{t_0}^t \big\| F_u \big\| \, d\tau.
\endaligned
\ee
\end{proposition}

The following energy estimates are due to Alinhac \cite{Alinhac1}, which are referred to as ghost weight energy estimates.
Compared with the natural energy estimates, an additional positive spacetime integral can also be controlled.

\begin{proposition}
Consider \eqref{eq:w-KG-m} with $m=0, 1$, and we have (with $\delta>0$) both
\bel{eq:gst-wKG1}
\aligned
\Ecal_{gst, m} (t, u)
\lesssim
\Ecal_{gst, m} (t_0, u)
+
\int_{t_0}^t \int_{\RR^2} \big| F_u \big| \, \big| \del_t u \big| \, dxd\tau
\endaligned
\ee
and
\bel{eq:gst-wKG2}
\aligned
\Ecal_{gst, m} (t, u)^{1/2}
\lesssim
\Ecal_{gst, m} (t_0, u)^{1/2}
+
\int_{t_0}^t \big\| F_u \big\| \, d\tau.
\endaligned
\ee
\end{proposition}

The following version of ghost weight energy estimates will play a vital role in the proof of the iteration procedure, which was used for instance in \cite{Dong2005}.

\begin{proposition}
Consider \eqref{eq:w-KG-m} with $m=1$, and it holds (with $\delta, \kappa > 0$)
\bel{eq:gst-wKG3}
\int_{t_0}^t \int_{\RR^2} {\kappa| u|^2 \over \langle \tau \rangle^{\kappa} \langle \tau-r\rangle^{1+\delta}} \, dxd\tau
\lesssim
\Ecal_{gst, 1} (t_0, u)
+
\int_{t_0}^t \int_{\RR^2} \la\tau \ra^{-\kappa} \big| F_u \big| \, \big| \del_t u \big| \, dxd\tau.
\ee
\end{proposition}
In the following application in Section \ref{sec:global-existence} we will take $\kappa = \delta/2$.

The above energy estimates are based on the following identity. Let $q(t,r): = \delta\int_{-\infty}^{r-t} \la s\ra^{-1-\delta}ds$ which is a uniformly bounded positive function. We apply the multiplier $\la t\ra^{-\kappa}e^q\del_tu$ and obtain:
$$
\aligned
\la t\ra^{-\kappa}e^q\del_tu\,\big(-\Box u + m^2u\big) 
=& \frac{1}{2}\del_t\Big(\la t\ra^{-\kappa}e^q\Big(\sum_\alpha |\del_{\alpha}u|^2 + m^2u^2\Big)\Big)
-\del_a\Big(\la t\ra^{\kappa}e^q\, \del_tu\del_au\Big)
\\
&+\frac{\delta}{2}\la t\ra^{-\kappa}e^q\la r-t\ra^{-1-\delta}\Big(\sum_a|G_au|^2 + m^2u^2\Big)
\\
&+ \frac{\kappa}{2}t\la t\ra^{-\kappa-2}e^q\Big(\sum_{\alpha}|\del_{\alpha}u|^2 + m^2u^2\Big).
\endaligned
$$
Then integrating the above identity in $\{0\leq \tau \leq t\}$ with Stokes' formula leads to the desired energy estimates. For \eqref{eq:natural-wKG1} and \eqref{eq:natural-wKG2} we take $\delta=0, e^q \equiv 1$ and $\kappa =0$. For \eqref{eq:gst-wKG1} we take $\delta>0$ and $\kappa =0$. For \eqref{eq:gst-wKG1} we fix $\delta,\kappa>0$.
 

\subsection{Estimates on commutators}

We first recall the well-known relations
$$
\Gamma \Box = \Box \Gamma,
\qquad
\Gamma (\Box -1) = (\Box -1) \Gamma,
$$
besides, we also need more estimates on commutators.
In order to apply the Klainderman-Sobolev type inequality, we need to bound quantities such as $\|\Gamma^I \del_{\alpha}u\|$ by the energies introduced in the last Subsection. For this purpose we first establish the following estimates on commutators.
\begin{lemma}
	For $u$ sufficiently regular, the following bounds hold:
	\begin{equation}\label{eq1-10-10-2021}
	\big|[\del_{\alpha},\Gamma^I] u\big|\leq C(I)\sum_{|J|<|I|}\sum_{\alpha'}|\del_{\alpha'}\Gamma^J u|
	\end{equation}
	\begin{equation}\label{eq2-10-10-2021}
	\big|[\del_{\alpha}\del_{\beta},\Gamma^I]u\big|\leq C(I)\sum_{|J|< |I|}\sum_{\alpha',\beta'}|\del_{\alpha'}\del_{\beta'}\Gamma^Ju|
	\end{equation}
	where $C(I)$ is a constant determined by $I$. When $|I| = 0$ the sums are understood to be zero.
\end{lemma}
\begin{proof}
	We need to establish the following decomposition:
	\begin{equation}\label{eq3-10-10-2021}
	[\del_{\alpha},\Gamma^I] = \sum_{ |J|<|I|}\sum_{}\pi_{\alpha J}^{I\beta}\del_{\beta}\Gamma^Ju
	\end{equation}
	where $\pi_{\alpha J}^{I \beta}$ are constants determined by $\alpha, I$. When $|I|=0$ the sum is understood to be zero. This can be checked by induction. First, when $|I| = 1$ and $\Gamma^I = \del_{\gamma}$, the commutators vanish. When $\Gamma^I = L_b$,
	$$
	[\del_t,L_b] = \del_b,\quad [\del_a,L_b] = \delta_{ab}\del_t
	$$
	which verify \eqref{eq3-10-10-2021}. Now suppose that \eqref{eq3-10-10-2021} holds for $|I|\leq k$ and we check the case with $|I| = k+1$. In this case we write
	$$
	[\del_{\alpha},\Gamma^I]u = [\del_{\alpha},\Gamma^{I_1}\Gamma^{I_2}]
	= [\del_{\alpha},\Gamma^{I_1}]\Gamma^{I_2}u +  \Gamma^{I_1}\big([\del_{\alpha},\Gamma^{I_2}]u\big).
	$$
	Then suppose that $|I_2|=1$. By the assumption of induction, 
	$$
	\aligned
	\,[\del_{\alpha},\Gamma^I]u 
	=& \sum_{|J_1|<|I_1|}\sum_{\beta}\pi_{\alpha J_1}^{I_1\beta}\del_{\beta}\Gamma^{J_1}\Gamma^{I_2}u 
	+ \sum_{\beta}\Gamma^{I_1}\big(\pi_{\alpha}^{I_2\beta}\del_{\beta} u \big) 
	\\
	=&\sum_{|J_1|<|I_1|}\sum_{\beta}\pi_{\alpha J_1}^{I_1\beta}\del_{\beta}\Gamma^{J_1}\Gamma^{I_2}u 
	+ \sum_{\beta}\pi_{\alpha}^{I_2\beta}\del_{\beta}\Gamma^{I_1} u
	- \sum_{\beta}\pi_{\alpha}^{I_2\beta}[\del_{\beta},\Gamma^{I_1}] u
	\\
	=&\sum_{|J_1|<|I_1|}\sum_{\beta}\pi_{\alpha J_1}^{I_1\beta}\del_{\beta}\Gamma^{J_1}\Gamma^{I_2}u 
	+ \sum_{\beta}\pi_{\alpha}^{I_2\beta}\del_{\beta}\Gamma^{I_1} u
	\\
	&+ \sum_{|J_1|<|I_1|}\sum_{\beta,\gamma}\pi_{\alpha}^{I_2\beta}\pi_{\beta J_1}^{I_1\gamma}\del_{\gamma}\Gamma^{J_1}.
	\endaligned
	$$
	Here remark that in the above three sums, $|J_1|+|I_2|<|I|, |I_1|<|I|$. This closes the induction. 
	\\
	For \eqref{eq2-10-10-2021}, we need the following decomposition
	\begin{equation}
	[\del_{\alpha}\del_{\beta},\Gamma^I]u
	 = \sum_{|J|<|I|}\pi_{\alpha\beta J}^{I\alpha'\beta'}\del_{\alpha'}\del_{\beta'}\Gamma^Ju
	\end{equation}
	where $\pi_{\alpha\beta J}^{I\alpha'\beta'}$ are constants determined by $\alpha,\beta,I$. This is by apply twice \eqref{eq3-10-10-2021}.
\end{proof}
%
%
%


\subsection{Global Sobolev inequalities} 
 
Recall that we do not commute the equations with the scaling vector field $L_0$ when studying the wave--Klein-Gordon systems, so we cannot directly apply the Klainerman-Sobolev inequality with the scaling vector field $L_0$. Thus we turn to the special version of the Klainerman-Sobolev inequality \eqref{eq:K-S} proved in \cite{Klainerman852}. The inequality is of vital importance in our study as the scaling vector field $L_0$ is excluded, even though we have certain price to pay: 1) to get the pointwise decay for a function at time $t>1$ we need the future information of the function till time $2t$; 2) we do not have any $\langle t-r \rangle$-decay of the function compared with the version of the Klainerman-Sobolev inequality with the scaling vector field $L_0$. Both of the flaws cause extra difficulties, and we need very delicate analysis to conquer them.

\begin{proposition}\label{prop:K-S}
Let $u = u(t, x)$ be a sufficiently smooth function which decays sufficiently fast at space infinity for each fixed $t \geq 0$.
Then for any $t \geq 0$, $x \in \RR^2$, we have
\bel{eq:K-S}
|u(t, x)|
\lesssim \langle t+|x| \rangle^{-1/2} \sup_{0\leq s \leq 2t, |I| \leq 3} \big\| \Gamma^I u(s) \big\|,
\qquad
\Gamma \in V = \{ L_a, \del_\alpha, \Omega_{ab} = x^a \del_b - x^b \del_a \}.
\ee
\end{proposition}

However, when we concentrate on the region outside of the light cone, the situation becomes less complicated. In fact we have the following version of global inequality:
\begin{proposition}\label{prop:S}
Let $u = u(t, x)$ be a sufficiently smooth function which decays sufficiently fast at space infinity for each fixed $t \geq 0$. 	Then 
\begin{equation}
|\la r-t\ra^{\eta} u(t,x)|\lesssim \la r\ra^{-1/2}\sum_{|I|\leq 2}\|\la r-t\ra^{\eta} \Lambda^I u(t,\cdot)\|
\end{equation}
where $\Lambda$ represents any of the vector in $\{\del_r,\Omega = x^1\del_2-x^2\del_1\}$ and $\Lambda^I$ be a product of these vectors with order $|I|$.
\end{proposition}
\begin{proof}[Sketch of proof]    
	This is a slightly modified version of the classical version (cf. for instance \cite{Klainerman852})
\begin{equation}\label{eq1-11-10-2021}
|u(t,x)|\leq \la r\ra^{-1/2} \sum_{|I|\leq 2}\|\Lambda^I u(t,\cdot)\|.
\end{equation}
We note $\Omega_{12}$ commutes with $r, t$, while we always get good terms when $\del_r$ acting on $\la r-t\ra^{\eta}$. Thus the proof is done.
\end{proof}


\subsection{Pointwise decay for Klein-Gordon components}
 
We recall that for linear homogeneous Klein-Gordon equations, the solutions decay at speed $\langle t+r \rangle^{-1}$ in $\RR^{1+2}$. Since the Kalinerman-Sobolev inequality in Proposition \ref{prop:K-S} gives at best $\langle t+r \rangle^{-1/2}$ decay rate for a given nice function in $\RR^{1+2}$, we need the following way to obtain optimal decay for Klein-Gordon components, which was introduced by Georgiev in \cite{Georgiev2}. 

We denote $\{ p_j \}_0^\infty$ a usual Paley-Littlewood partition of the unity
$$
1 = \sum_{j \geq 0} p_j(s),
\qquad
s \geq 0,
$$
which is assumed to satisfy 
$$
0 \leq p_j \leq 1,
\qquad
p_j \in C_0^\infty (\RR), 
\qquad
\text{for all $j \geq 0$},
$$
and the supports of the series satisfy
$$
\text{supp } p_0 \subset (-\infty, 2],
\qquad
\text{supp } p_j \subset [2^{j-1}, 2^{j+1}],
\qquad
\text{for all $j \geq 1$}.
$$

Now we are ready to give the statement of the decay result for the Klein-Gordon equations in \cite{Georgiev2}.
\begin{proposition}\label{prop:G}
Let $v$ solve the Klein-Gordon equation
$$
- \Box v + v = f,
$$
with $f = f(t, x)$ a sufficiently nice function.
Then for all $t \geq 0$, it holds
\be 
\aligned
&\langle t + |x| \rangle |v(t, x)|
\\
\lesssim
&\sum_{j\geq 0,\, |I| \leq 4} \sup_{0\leq s \leq t} p_j(s) \big\| \langle s+|x| \rangle \Gamma^I f(s, x) \big\|
+
\sum_{j\geq 0,\, |I| \leq 4} \big\| \langle |x| \rangle p_j (|x|) \Gamma^I v(0, x) \big\|
\endaligned
\ee

\end{proposition}

As a consequence, we have the following simplified version of Proposition \ref{prop:G}.

\begin{proposition}\label{prop:G1}
With the same settings as Proposition \ref{prop:G}, let $\delta' > 0$ and assume 
$$
 \sum_{|I| \leq 4} \big\| \langle s+|x| \rangle \Gamma^I f(s, x) \big\|
\leq  C_f \langle s \rangle^{-\delta'} ,
$$
then we have
\be 
\langle t + |x| \rangle |v(t, x)| 
\lesssim
{C_f \over 1-2^{-\delta'} }
+
\sum_{|I| \leq 4} \big\| \langle |x| \rangle \log\langle |x| \rangle  \Gamma^I v(0, x) \big\|.
\ee
\end{proposition}


\section{Extra decay for wave and Klein-Gordon components}\label{sec:extra-decay}

In the analysis, we will distinguish the regions $\{(t, x) : |x| \leq 2t \}$ and $\{(t, x) : |x| \geq 2t \}$, because we can take advantage of the extra decay properties of wave and Klein-Gordon components away from the light cone. So the extra decay properties of wave and Klein-Gordon components in the propositions below will also be demonstrated differently in different spacetime regions.

\subsection{Extra decay for Hessian of wave components}

\begin{proposition}\label{prop:wave-extra}
Consider the wave equation
$$
-\Box w = F_w,
$$
then we have
\bel{eq:w-Hessian1}
|\del \del w|
\lesssim
{1\over \langle t-r \rangle} \big(|\del \Gamma w| + |\del w| \big) + {t \over \langle t-r \rangle} |F_w|,
\qquad
\text{for } |x| \leq 3t,
\ee
as well as
\bel{eq:w-Hessian2}
|\del \del w|
\lesssim
{r\over \langle t\rangle^2} \big(|\del \Gamma w| + |\del w| \big) + |F_w|,
\qquad
\text{for } |x| \geq 3t/2,
\ee
\end{proposition}
\begin{proof}
For completeness we revisit the proof in \cite{PLF-YM-book}. Since it is easily seen that the results hold for $t \leq 1$, so we will only consider the case $t \geq 1$.

We first express the wave operator $- \Box$ by $\del_t, L_a$ to get
\bel{eq:wave-2expression}
\aligned
-\Box
=
{(t-|x|) (t+ |x|) \over t^2} \del_t\del_t
+ {x^a \over t^2} \del_t L_a
- {1\over t} \del^a L_a
+ {2 \over t} \del_t
-  {x^a \over t^2} \del_a.
\endaligned
\ee
When $t\geq 1$, one has
$$
| \del_t\del_t w |
\lesssim
{1\over \la r-t\ra} \big( \big| \del \Gamma w \big| + \big| \del w \big|  \big)
+
{t^2\over \la r-t\ra\la r+t\ra}|F_w|.
$$

On the other hand, we note that the following relations hold true
$$
\aligned
\del_a \del_t 
&= - {x_a \over t} \del_t \del_t + {1\over t} \del_t L_a  - {1\over t} \del_a,
\\
\del_a \del_b
&= {x_a x_b \over t^2} \del_t\del_t 
- {x_a \over t^2} \del_t L_b 
+ {1\over t} \del_b L_a 
- {\delta_{ab} \over t} \del_t
+ {x_a \over t^2} \del_b,
\endaligned
$$
This leads to
$$
|\del_{\alpha}\del_{\beta}w|\lesssim \frac{r^2}{\la t\ra^2\la r-t\ra}\big(|\del\Gamma w| + |\del w|\big) + \frac{r|F_w|}{\la r-t\ra}.
$$
Now when $r\leq3t$, we remark that $r\leq \la t\ra$. Then the above bound reduces to \eqref{eq:w-Hessian1}. When $r\geq 3t/2$, $\la r+t\ra\lesssim \la r-t\ra$. The above bound reduces to \eqref{eq:w-Hessian2}.

%
%
%

%
\end{proof}

\subsection{Extra decay for wave components}

We recall the smooth and increasing function defined in \eqref{eq:cut-off}
\begin{eqnarray}\notag
\chi(s):=
\left\{
\begin{array}{lll}
0, & \quad s\leq 1,
\\
1, &\quad s \geq 2.
\end{array}
\right.
\end{eqnarray}

\begin{proposition}\label{prop:wave-extra2}
	Consider the wave equation
	$$
	-\Box w = F_w,
	$$
	then we have
	\bel{eq:wave-extra1}
	\aligned
	&\int_{\RR^2} \chi (r-t) \langle r-t \rangle^{2\eta} (\del w)^2 \, dx
	\\
	\lesssim
	&\big\| \langle r \rangle^{ \eta} \del w (0, \cdot) \big\|^2 
	+
	\int_0^t \int_{\RR^2} \chi(r-\tau) \langle r-\tau \rangle^{2\eta} \big| F_w \del_t w \big| \, dxd\tau.
	\endaligned
	\ee
	
\end{proposition}

\begin{proof}
	We pick $\chi(r-t) \langle r-t \rangle^{2\eta} \del_t w$ as the multiplier, and we derive the identity
	$$
	\aligned
	&{1\over 2} \del_t \Big( \chi(r-t) \langle r-t \rangle^{2\eta} \big( (\del_t w)^2 + \sum_a (\del_a w)^2 \big) \Big)
	-
	\del_a \big( \chi(r-t)  \langle r-t \rangle^{2\eta} \del^a w \del_t w \big)
	\\
	+
	&\frac{1}{2}\big(\chi'(r-t) \langle r-t \rangle^{2\eta} + 2\eta\chi(r-t)\la r-t\ra^{2\eta-2}(r-t)\big)\sum_a|G_au|^2 
	\\
	=
	&\chi(r-t)  \langle r-t \rangle^{2\eta} F_w \del_t w.
	\endaligned
	$$
	We observe that
	$$
	\big(\chi'(r-t) \langle r-t \rangle^{2\eta} + 2\eta\chi(r-t)\la r-t\ra^{2\eta-2}(r-t)\big)\sum_a|G_au|^2
	\geq 0.
	$$
	Then we are led to the desired energy estimates \eqref{eq:wave-extra1} by integrating the above identity over the spacetime region $[0, t] \times \RR^2$.
	
	The proof is complete.
\end{proof}

The following energy estimates allow us, in many cases, to gain better $t$-bound for the energy of the wave component at the expense of losing some $\langle t-r \rangle$-bound inside of the light cone $\{ r\leq t \}$. This idea was applied for instance in \cite{Dong2006} when studying the Klein-Gordon-Zakharov equation in two space dimensions with compactness assumptions, and is now adapted to the non-compact setting.

\begin{proposition}\label{prop:wave-extra3}
Consider the wave equation
$$
-\Box w = F_w.
$$
We have
\bel{eq:EE-wave}
\aligned
&\int_{\RR^2} \Big( (t-r)^{-\eta} \chi(t-r) + 1 - \chi(t-r) \Big) \big| \del w \big|^2 \, dx
\\
\lesssim
&\Ecal_{gst} (t_0, w) 
+
\int_{t_0}^t \Big( (\tau-r)^{-\eta} \chi(\tau-r) + 1 - \chi(\tau-r) \Big) \big| F_w \del_t w\big| \, dxd\tau.
\endaligned
\ee
\end{proposition}
\begin{proof}
Consider the $w$ equation, and take the multiplier 
$$
\big( (t-r)^{-\eta} \chi(t-r) + 1 - \chi(t-r) \big) \del_t w
$$
to have the differential identity
$$
\aligned
&{1\over 2} \del_t \Big( \big( (t-r)^{-\eta} \chi(t-r) + 1 - \chi(t-r)  \big) \big( (\del_t w)^2 + \sum_a (\del_a w)^2  \big)   \Big)
\\
- &\del_a \Big(  \big( (t-r)^{-\eta} \chi(t-r) + 1 - \chi(t-r)  \big) \del^a w \del_t w \Big)
\\
+& {\eta \over 2} (t-r)^{-\eta-1} \chi(t-r) \sum_a|G_au|^2
+ {1\over 2} \big( \chi'(t-r) - (t-r)^{-\eta} \chi'(t-r)  \big) \sum_a|G_au|^2
\\
= 
&\big( (t-r)^{-\eta} \chi(t-r) + 1 - \chi(t-r) \big) \del_t w F_w.
\endaligned
$$
We note that
$$
{\eta \over 2} (t-r)^{-\eta-1} \chi(t-r) \sum_a|G_au|^2
+ {1\over 2} \big( \chi'(t-r) - (t-r)^{-\eta} \chi'(t-r)  \big) \sum_a|G_au|^2
\geq 0,
$$
so we have
$$
\aligned
&{1\over 2} \del_t \Big( \big( (t-r)^{-\eta} \chi(t-r) + 1 - \chi(t-r)  \big) \big( (\del_t w)^2 + \sum_a (\del_a w)^2  \big)   \Big)
\\
- &\del_a \Big(  \big( (t-r)^{-\eta} \chi(t-r) + 1 - \chi(t-r)  \big) \del^a w \del_t w \Big)
\\
\leq 
&\big( (t-r)^{-\eta} \chi(t-r) + 1 - \chi(t-r) \big) \del_t w F_w.
\endaligned
$$
Integrating this inequality over the region $[t_0, t] \times \RR^2$ yields the desired result.

The proof is done.
\end{proof}

\subsection{Extra decay for Klein-Gordon components}

\begin{proposition}\label{prop:KG-extra}
Consider the Klein-Gordon equation
$$
-\Box v + v = F_v,
$$
then for $t\geq 1$ we have
\bel{eq:KG-extra1}
|v| \lesssim
{|t-r| \over \langle t \rangle} |\del \del v| + {1\over \langle t \rangle} |\del \Gamma v| + {1\over \langle t \rangle} |\del v| + |F_v|,
\qquad
\text{for } |x| \leq 3t.
\ee

\end{proposition}

\begin{proof}
This phenomena was detected in \cite{Hormander, Klainerman93}. The following proof can be found in \cite{Ma19} in the light cone. Here we give a generalization in a larger region of spacetime.
Recall the expression of the wave operator in \eqref{eq:wave-2expression}, we find
$$
\aligned
-\Box v + v 
= 
{(t-|x|) (t+ |x|) \over t^2} \del_t\del_t v
+ {x^a \over t^2} \del_t L_a v
- {1\over t} \del^a L_a v
+ {2 \over t} \del_t v
-  {x^a \over t^2} \del_a v + v,
\endaligned
$$
which leads us to
$$
|v|
\lesssim
 {|(t-|x|) (t+ |x|)| \over t^2} |\del_t\del_t v|
+ {|x^a| \over t^2} |\del_t L_a v|
+ {1\over t} |\del^a L_a v|
+ {2 \over t} |\del_t v|
+  {|x^a| \over t^2} |\del_a v|
+ |F_v|.
$$
If $|x| \leq 3t$, we further have
$$
|v|
\lesssim
{|t-|x|| \over t} |\del_t\del_t v|
+ \sum_a {1 \over t} |\del L_a v|
+ {1 \over t} |\del v|
+ |F_v|,
$$
which finishes the proof.
\end{proof}

We recall the smooth and increasing function $\chi$ defined in \eqref{eq:cut-off}.

\begin{proposition}\label{prop:KG-extra2}
Consider the Klein-Gordon equation
$$
-\Box v + v = F_v,
$$
then for all $\eta \geq 0$ we have
\bel{eq:KG-extra2}
\aligned
&\int_{\RR^2} \chi (r-t) \langle r-t \rangle^{2\eta} \big( (\del v)^2 + v^2 \big) \, dx
\\
+ & \int_0^t\int_{\RR^2} \Big( \chi'(r-t) \langle r-t \rangle^{2\eta}
+ \eta \chi(r-t)  \langle r-t \rangle^{2\eta-2} (r-t)\Big)  v^2  \, dx\, d\tau
\\
\lesssim
&\big\| \langle r \rangle^{ \eta} \del v (0, \cdot) \big\|^2 + \big\| \langle r \rangle^{ \eta}  v(0, \cdot) \big\|^2
+
\int_0^t \int_{\RR^2} \chi(r-\tau) \langle r-\tau \rangle^{2\eta} \big| F_v \del_t v \big| \, dxd\tau.
\endaligned
\ee
\end{proposition}

\begin{proof}
The proof is almost the same as the proof of Proposition \ref{prop:wave-extra2}. We choose
$$
\chi(r-t) \langle r-t \rangle^{2\eta} \del_t v
$$
to be the multiplier, and we derive the identity
$$
\aligned
&{1\over 2} \del_t \Big( \chi(r-t) \langle r-t \rangle^{2\eta} \big( (\del_t v)^2 + \sum_a (\del_a v)^2 + v^2 \big) \Big)
-
\del_a \big( \chi(r-t)  \langle r-t \rangle^{2\eta} \del^a v \del_t v \big)
\\
+
&\frac{1}{2}\big(\chi'(r-t) \langle r-t \rangle^{2\eta} + 2\eta \chi(r-t)  \langle r-t \rangle^{2\eta-2} (r-t)\big) \Big(\sum_a|G_a v| + v^2\Big)
\\
=
&\chi(r-t)  \langle r-t \rangle^{2\eta} F_v \del_t v.
\endaligned
$$
We observe that
$$
\big(\chi'(r-t) \langle r-t \rangle^{2\eta} + 2\eta \chi(r-t)  \langle r-t \rangle^{2\eta-2} (r-t)\big) \sum_a|G_a v|\geq 0.
\\
$$
Then we are led to the desired energy estimates \eqref{eq:KG-extra2} by integrating the above identity over the spacetime region $[0, t] \times \RR^2$.

The proof is complete.
\end{proof}


\section{Global existence}\label{sec:global-existence}

\subsection{Solution space and solution mapping}

In many cases, to prove the global existence of a nonlinear system one relies on a bootstrap argument. In our case, due to the utilisation of the Klainerman-Sobolev inequality without scaling vector field stated in Proposition \ref{prop:K-S}, where one requires the future information till time $2t$ when deriving the pointwise estimate for the function at time $t$, we turn to the aid of the contraction mapping theorem, and thus an iteration procedure. For easy readability, we will use capital letters (like $\Psi, V$) to denote Klein-Gordon components, while small letters (like $\phi, u$) are used to represent wave components.

We first define the solution space $X$ with some small $0< \delta \ll 1$ (recall the regularity index $N\geq 14$ below). We recall the cut-off function $\chi$, which is smooth and increasing, defined in \eqref{eq:cut-off}
$$
\chi(s):=
\left\{
\begin{array}{lll}
0, & \quad s\leq 1,
\\
1, &\quad s \geq 2.
\end{array}
\right.
$$

\begin{definition}\label{def:X}
Let $\Psi = \Psi (t, x), \phi = \phi (t, x)$ be sufficiently regular functions, in which $\Psi$ is an $\RR^2$-valued function while $\phi$ is a scalar-valued function, and we say $(\Psi, \phi)$ belongs to the metric space $X$ if
\bei
\item It satisfies
\be
\big( \Psi, \del_t \Psi, \phi, \del_t \phi \big) (0, \cdot)
= \big(E_0, E_1, n_0, n_1\big). 
\ee

\item It satisfies
\be
\big\| (\Psi, \phi ) \big\|_X
\leq C_1 \eps,
\ee
in which $C_1 \gg 1$ is some big constant to be determined, the size of the initial data $\eps \ll 1$ is small enough such that $C_1 \eps \ll 1$, and the norm $\| \cdot \|_X$ for a pair of $\RR^2\times \RR$-valued functions $(V, u)$ is defined by (with $0<\delta\ll 1$)
\bel{eq:X-norm}
\aligned
\big\| (V, u) \big\|_X
:=
&\sup_{t \geq 0,\, |I| \leq N+1} \langle t \rangle^{-\delta} \big( \big\| \Gamma^I u \big\| + \Ecal_{gst, 1} ( t, \Gamma^I V)^{1/2} \big) 
\\
+
&\sup_{t \geq 0,\, |I| \leq N+1} \langle t \rangle^{-\delta/2} \Big(\int_0^t \Big\| {\Gamma^I V \over \langle\tau\rangle^{\delta/2} \langle \tau-r \rangle^{1/2+\delta/2}} \Big\|^2 \, d\tau \Big)^{1/2}
\\
+
&\sup_{t \geq 0,\, |I| \leq N} \langle t \rangle^{-\delta}  \Big\| {\langle t+r \rangle \over \langle t-r \rangle} \Gamma^I V \Big\|
\\
+
&\sup_{t \geq 0,\, |I| \leq N-2} \Ecal_{gst} ( t, \Gamma^I u)^{1/2} + \sup_{t \geq 0,\, |I| \leq N-1} \Ecal_{gst, 1} ( t, \Gamma^I V)^{1/2} 
\\
+
&\sup_{t \geq 0,\, |I| \leq N} \langle t \rangle^{-\delta} \big\| \big( 1-\chi(r-2t) \big)^{1/2} \langle t-r \rangle \Gamma^I u \big\|
\\
+
&\sup_{t \geq 0,\, |I| \leq N-2}  \big\| \big( 1-\chi(r-2t) \big)^{1/2} \langle t-r \rangle^{1-\delta} \Gamma^I u \big\|
\\
+
&\sup_{t \geq 0,\, |I| \leq N-1} \langle t \rangle^{-1/2-\delta} \big\| \chi^{1/2}(r-t) \langle t-r \rangle \Gamma^I u \big\|
\\
+
&\sup_{t \geq 0,\, |I| \leq N-5} \langle t \rangle^{-\delta} \big\| \chi^{1/2}(r-t) \langle t-r \rangle \Gamma^I u \big\|
\\
+
&\sup_{t \geq 0,\, |I| \leq N-1} \langle t \rangle^{-1/2-\delta} \big(\big\| \chi^{1/2}(r-t) \langle t-r \rangle \del \Gamma^I V \big\|
+ \big\| \chi^{1/2}(r-t) \langle t-r \rangle \Gamma^I V \big\| \big)
\\
+
&\sup_{t \geq 0,\, |I| \leq N-5} \langle t \rangle^{-\delta} \big(\big\| \chi^{1/2}(r-t) \langle t-r \rangle \del \Gamma^I V \big\|
+ \big\| \chi^{1/2}(r-t) \langle t-r \rangle \Gamma^I V \big\| \big)
\\
+
&\sup_{t, r \geq 0,\, |I| \leq N-5} \langle t+r \rangle \big|\Gamma^I V \big| 
+ \sup_{t, r \geq 0,\, |I| \leq N-7} \langle t-r \rangle^{-1} \langle t+r \rangle^{2} \big|\Gamma^I V \big|
\\
+
&
\sup_{t, r \geq 0,\, |I| \leq N-8} \Big( \langle t-r \rangle^{1-\delta} \langle t+r \rangle^{1/2} \big|\Gamma^I u \big| 
+ \chi(r-t) \langle t+r\rangle^{5/4-\delta} \big|\Gamma^I V \big| \Big).
\endaligned
\ee

\eei

\end{definition}

It is easy to see that the solution space $X$ is complete with respect to the metric induced from the norm $\|\cdot\|_X$. Next, we want to construct a contraction mapping. To achieve this, we first define a solution mapping, and then prove it is also a contraction mapping by carefully choosing the size of the parameters $C_1, \eps$. We recall that $A \lesssim B$ means $A \leq C B$ with $C$ independent of $C_1, \eps$.

\begin{definition}
Given a pair of functions $(\Psi, \phi) \in X$, the solution mapping $T$ maps it to the unique pair of functions $\big(\widetilde{\Phi}, \widetilde{\phi}\big)$, which is the solution to the following linear equations
\bel{eq:solution-map}
\aligned
&-\Box \widetilde{\Psi} + \widetilde{\Psi} = - \phi \Psi,
\\
&-\Box \widetilde{\phi} = \Delta |\Psi|^2,
\\
&\big(\widetilde{\Phi}, \del_t \widetilde{\Phi}, \widetilde{\phi}, \del_t \widetilde{\phi} \big)(t_0)
=
(E_0, E_1, n_0, n_1),
\endaligned
\ee
and we will write $\big(\widetilde{\Psi}, \widetilde{\phi}\big) = T (\Psi, \phi)$.

\end{definition}



\subsection{Contraction mapping and global existence}

The goal of this part is to show the solution mapping $T$ is a contraction mapping from the solution space $X$ to $X$.

\begin{proposition}\label{prop:mapping1}
With suitably chosen large $C_1$ and small $\eps$, we have the following.
\bei
\item Given a pair of functions $(\Psi, \phi) \in X$, we have
\bel{eq:contraction1}
T (\Psi, \phi) \in X.
\ee

\item For any $(\Psi, \phi), (\Psi', \phi') \in X$, it holds
\bel{eq:contraction2}
\big\|T(\Psi, \phi) - T(\Psi', \phi') \big\|_X
\leq {1\over 2} \big\|(\Psi, \phi) - (\Psi', \phi') \big\|_X.
\ee
\eei
\end{proposition}

We rewrite \eqref{eq:solution-map} to take advantage of the special structure (i.e., a hidden divergence form structure) of the nonlinearities appearing in the wave equation of $\widetilde{\phi}$, which read
\bel{eq:solution-map1}
\aligned
&-\Box \widetilde{\Psi} + \widetilde{\Psi} = - \phi \Psi,
\\
&-\Box \widetilde{\phi}^\Delta =  |\Psi|^2,
\qquad\qquad
\widetilde{\phi} = \Delta \widetilde{\phi}^\Delta,
\\
&\big(\widetilde{\Phi}, \del_t \widetilde{\Phi}, \widetilde{\phi}^\Delta, \del_t \widetilde{\phi}^\Delta \big)(t_0)
=
(E_0, E_1, n_0^\Delta, n_1^\Delta).
\endaligned
\ee
We note this kind of reformulation has been used before; see for instance \cite{Katayama12a}.
To estimate higher-order energy, we act $\Gamma^I$ to \eqref{eq:solution-map1} to get
\bel{eq:solution-map2}
\aligned
&-\Box \Gamma^I \widetilde{\Psi} + \Gamma^I \widetilde{\Psi} = - \Gamma^I \big(\phi \Psi\big),
\\
&-\Box \Gamma^I \widetilde{\phi}^\Delta = \Gamma^I \big(|\Psi|^2\big).
\endaligned
\ee

In the sequel, we will write $\big(\widetilde{\Psi}, \widetilde{\phi}\big) = T (\Psi, \phi)$.

\begin{lemma}\label{lem:est-00a}
If $(\Psi, \phi)$ lies in the solution space $X$, then the following estimates hold
\bel{eq:est-00a}
\aligned
{\la t+r \ra \over \la t-r\ra} \big|  \Gamma^I \Psi \big|
\leq
C_1 \eps \la t+r\ra^{-1},
\qquad
|I| \leq N-7,
\\
\la t-r \ra \big|  \Gamma^I \phi \big|
\leq
C_1 \eps \la t+r\ra^{-1/2+\delta},
\qquad
|I| \leq N-8,
\\
{\la t-r \ra \over \la t+r\ra} \big|  \Gamma^I \phi \big|
\leq
C_1 \eps \la t+r\ra^{-3/2+\delta},
\qquad
|I| \leq N-8.
\endaligned
\ee
\end{lemma}

\begin{lemma}\label{lem:est-001}
Let $(\Psi, \phi)$ lie in the solution space $X$, then we have
\bel{eq:est-001}
\aligned
\Ecal_{gst, 1} (t, \Gamma^I \widetilde{\Psi})^{1/2}
\lesssim
\eps + (C_1 \eps)^{3/2} \langle t \rangle^\delta,
\qquad
|I| \leq N+1,
\endaligned
\ee
\bel{eq:est-002}
\Big(\int_0^t \Big\| {\Gamma^I \widetilde{\Psi} \over \langle\tau\rangle^{\delta/2} \langle \tau-r \rangle^{1/2+\delta/2}} \Big\|^2 \, d\tau \Big)^{1/2}
\lesssim
\eps + (C_1 \eps)^{3/2} \langle t \rangle^{\delta/2},
\qquad
|I| \leq N+1.
\ee
\end{lemma}
\begin{proof}
We first show \eqref{eq:est-001}. Consider \eqref{eq:solution-map2} and apply the ghost weight energy estimates \eqref{eq:gst-wKG2}, and for $|I|\leq N+1$ we find
$$
\aligned
\Ecal_{gst, 1} (t, \Gamma^I \widetilde{\Psi})^{1/2}
\lesssim
\Ecal_{gst, 1} (t_0, \Gamma^I \widetilde{\Psi})^{1/2}
+
\int_0^t \big\| \Gamma^I \big(\phi \Psi\big) (\tau) \big\| \, d\tau.
\endaligned
$$
Recall that Leibniz rule yields
$$
\Gamma^I \big(\phi \Psi\big)
=
\sum_{I_1 + I_2 = I} \Gamma^{I_1} \phi \, \Gamma^{I_2} \Psi,
$$
and we thus have
\bel{eq:est-001a}
\aligned
&\big\| \Gamma^I \big(\phi \Psi\big) \big\|
\\
\lesssim
&\sum_{|I_1| + |I_2| = |I|} \big\| \Gamma^{I_1} \phi \, \Gamma^{I_2} \Psi \big\|
\\
\lesssim
&\sum_{\substack{|I_2| \leq N-5\\ |I_1| + |I_2| \leq |I|}} \big\| \Gamma^{I_1} \phi \, \Gamma^{I_2} \Psi \big\|
+
\sum_{\substack{|I_1| \leq N-8\\ |I_1| + |I_2| \leq |I|}} \big\| \Gamma^{I_1} \phi \, \Gamma^{I_2} \Psi \big\|,
\endaligned
\ee
in which we used $N\geq 14$ in the last inequality.

Next, we estimate those two quantities in the above inequality. We start with
\bel{eq:001b}
\aligned
\sum_{\substack{|I_2| \leq N-5\\ |I_1| + |I_2| \leq |I|}} \big\| \Gamma^{I_1} \phi \, \Gamma^{I_2} \Psi \big\|
\lesssim
\sum_{\substack{|I_2| \leq N-5\\ |I_1|  \leq |I|}} \big\| \Gamma^{I_1} \phi \big\| \, \big\| \Gamma^{I_2} \Psi \big\|_{L^\infty}
\lesssim
(C_1 \eps)^2 t^{-1+\delta}.
\endaligned
\ee
As for the other one, we have
\bel{eq:001c}
\aligned
&\sum_{\substack{|I_1| \leq N-8\\ |I_1| + |I_2| \leq |I|}} \big\| \Gamma^{I_1} \phi \, \Gamma^{I_2} \Psi \big\|
\\
\lesssim
&\sum_{\substack{|I_1| \leq N-8\\ |I_1| + |I_2| \leq |I|}} \big\|\langle t\rangle^{\delta/2} \langle t-r \rangle^{1/2+\delta/2} \Gamma^{I_1} \phi \big\|_{L^\infty} \, \Big\| {\Gamma^{I_2} \Psi \over \langle t\rangle^{\delta/2} \langle t-r \rangle^{1/2+\delta/2}} \Big\|
\\
\lesssim
&C_1 \eps \langle t \rangle^{-1/2+\delta/2} \sum_{|I_2|\leq N+1}\Big\| {\Gamma^{I_2} \Psi \over \langle t\rangle^{\delta/2} \langle t-r \rangle^{1/2+\delta/2}} \Big\|.
\endaligned
\ee

Gathering the estimates leads us to \eqref{eq:est-001}
$$
\aligned
&\Ecal_{gst, 1} (t, \Gamma^I \widetilde{\Psi})^{1/2}
\\
\lesssim
&\eps 
+
\int_0^t \Big( (C_1 \eps)^2 \tau^{-1+\delta} + C_1 \eps \langle \tau \rangle^{-1/2+\delta/2} \sum_{|I_2|\leq N+1} \Big\| {\Gamma^{I_2} \Psi \over \langle\tau\rangle^{\delta/2} \langle \tau-r \rangle^{1/2+\delta/2}} \Big\| \Big) \, d\tau
\\
\lesssim
&\eps + (C_1 \eps)^2 \langle t \rangle^\delta
+
C_1 \eps \Big( \int_0^t \langle \tau \rangle^{-1+\delta} \, d\tau\Big)^{1/2}  \sum_{|I_2|\leq N+1} \Big( \int_0^t \Big\| {\Gamma^{I_2} \Psi \over \langle\tau\rangle^{\delta/2} \langle \tau-r \rangle^{1/2+\delta/2}} \Big\|^2 \, d\tau \Big)^{1/2}
\\
\lesssim
&\eps + (C_1 \eps)^2 \langle t \rangle^\delta.
\endaligned
$$

Finally, we want to deduce \eqref{eq:est-002}.
The ghost weight energy estimates \eqref{eq:gst-wKG3} indicate
$$
\aligned
\int_0^t \Big\| {\Gamma^I \widetilde{\Psi} \over \tau^{\delta/2} \langle \tau-r \rangle^{1/2+\delta/2}} \Big\|^2 \, d\tau 
\lesssim
\Ecal_{gst, 1} (t_0, \Gamma^I \widetilde{\Psi})
+
\int_0^t \int_{\RR^2} \langle \tau \rangle^{-\delta} \big| \Gamma^I \big(\phi \Psi\big) \del_t \Gamma^I \widetilde{\Psi} \big|(\tau, x) \, dxd\tau.
\endaligned
$$
With the estimates \eqref{eq:est-001} we just proved, we have
$$
\aligned
&\int_0^t \int_{\RR^2} \langle \tau \rangle^{-\delta} \big| \Gamma^I \big(\phi \Psi\big) \del_t \Gamma^I \widetilde{\Psi} \big|(\tau, x) \, dxd\tau.
\\
\lesssim
&
\int_0^t \langle \tau \rangle^{-\delta} \big\| \Gamma^I \big(\phi \Psi\big) \big\| \big\| \del_t \Gamma^I \widetilde{\Psi} \big\| \, d\tau
\\
\lesssim
&
\Big(\int_0^t \langle \tau \rangle^{-\delta+1} \big\| \Gamma^I \big(\phi \Psi\big) \big\|^2 \, d\tau \Big)^{1/2}
\Big( \int_0^t \langle \tau \rangle^{-\delta-1} \big\| \del_t \Gamma^I \widetilde{\Psi} \big\|^2 \, d\tau \Big)^{1/2}
\\
\lesssim
&
C_1 \eps \langle t \rangle^{\delta/2} \Big(\int_0^t \langle \tau \rangle^{-\delta+1} \big\| \Gamma^I \big(\phi \Psi\big) \big\|^2 \, d\tau \Big)^{1/2}.
\endaligned
$$
To proceed, we find
$$
\aligned
&\Big(\int_0^t \langle \tau \rangle^{-\delta+1} \big\| \Gamma^I \big(\phi \Psi\big) \big\|^2 \, d\tau \Big)^{1/2}
\\
\lesssim
&\Big( \int_0^t \Big( (C_1 \eps)^2 \langle \tau\rangle^{-1+\delta} + C_1 \eps \sum_{|I_2|\leq N+1} \Big\| {\Gamma^{I_2} \Psi \over \langle \tau \rangle^{\delta/2} \langle \tau-r \rangle^{1/2+\delta/2}} \Big\|^2 \Big) \, d\tau \Big)^{1/2}
\\
\lesssim
& (C_1 \eps)^2 \langle t \rangle^{\delta /2},
\endaligned
$$
which further gives us
$$
\aligned
\int_0^t \Big\| {\Gamma^I \widetilde{\Psi} \over \tau^{\delta/2} \langle \tau-r \rangle^{1/2+\delta/2}} \Big\|^2 \, d\tau 
\lesssim
\eps^2 + (C_1 \eps)^3 \langle t \rangle^{\delta},
\endaligned
$$
and hence \eqref{eq:est-002}.

We complete the proof.
\end{proof}

\begin{lemma}\label{lem:est-005a}
The following estimates are valid
\bel{eq:est-005a}
\aligned
\Big\| {\langle t+r \rangle \over \langle t-r \rangle} \Gamma^I \widetilde{\Psi} \Big\|
\lesssim
\eps+ (C_1 \eps)^{3/2} \langle t\rangle^\delta,
\qquad
|I| \leq N.
\endaligned
\ee
\end{lemma}
\begin{proof}
In the region $\{ r\geq 2t \}$, it holds $\langle t+r \rangle \lesssim \langle t-r \rangle$, thus we only need to consider the region $\{ r\leq 2t \}$. Recall that Lemma \ref{lem:est-001} provides us with
$$
\big\| \del \Gamma^J \widetilde{\Psi} \big\|
+
\big\| \Gamma^J \widetilde{\Psi} \big\|
\lesssim 
\eps + (C_1 \eps)^2 \langle t\rangle^\delta,
\qquad
|I| \leq N+1.
$$ 
Thus the proof for the region $\{ r\leq 2t \}$ follows from Lemma \ref{lem:est-001} and Proposition \ref{prop:KG-extra}.
\end{proof}

\begin{lemma}\label{lem:est-005b}
The following estimates hold
\bel{eq:est-005b}
\big| \Gamma^I \widetilde{\Psi} \big|
\lesssim
\big(\eps + (C_1 \eps)^2 \big) \langle t+r \rangle^{-1},
\qquad
|I| \leq N-5.  
\ee
\end{lemma}
\begin{proof}
By Proposition \ref{prop:G1}, we need to bound the quantity
$$
\sum_{|I| \leq N-1}\big\| \langle t+r\rangle \Gamma^I \big( \phi \Psi \big) \big\|.
$$
We have
$$
\aligned
&\sum_{|I| \leq N-1}\big\| \langle t+r\rangle \Gamma^I \big( \phi \Psi \big) \big\|
\\
\lesssim
&\sum_{\substack{|I_1|\leq N-1\\ |I_2| \leq N-8}} \Big\|  {\langle t+r \rangle \over \langle t-r \rangle} \Gamma^{I_1} \Psi  \Big\|   
   \big\| \langle t-r \rangle \Gamma^{I_2} \phi \big\|_{L^\infty}
+ 
\sum_{\substack{|I_1|\leq N-7\\ |I_2| \leq N-1}}  \Big\|  {\langle t+r \rangle \over \langle t-r \rangle} \Gamma^{I_1} \Psi  \Big\|_{L^\infty}    
   \big\| \langle t-r \rangle \Gamma^{I_2} \phi \big\|
\\
\lesssim
&(C_1 \eps)^2 \langle t \rangle^{-1/2+3\delta}.  
\endaligned
$$
Then Proposition \ref{prop:G1} yields the desired result \eqref{eq:est-005b}.
\end{proof}

\begin{lemma}\label{lem:est-005c}
It holds
\bel{eq:est-005c}
\big| \Gamma^I \widetilde{\Psi} \big|
\lesssim
\big( \eps + (C_1 \eps)^2 \big) {\langle t-r \rangle \over \langle t+r \rangle^2},
\qquad
|I| \leq N-7. 
\ee
\end{lemma}
\begin{proof}
The proof follows from Lemma \ref{lem:est-005b} and Proposition \ref{prop:KG-extra}.
\end{proof}

\begin{lemma}
We have  
\bel{eq:est-020}
\aligned
\big\| \chi^{1/2}(r-t) \langle r-t \rangle \del \Gamma^I \widetilde{\Psi} \big\|
&+
\big\| \chi^{1/2}(r-t) \langle r-t \rangle \Gamma^I \widetilde{\Psi} \big\|
\\
\lesssim&
\left\{
\aligned
&\eps+(C_1 \eps)^2 \langle t \rangle^{1/2+\delta},
&&|I| \leq N-1,
\\
&\eps+(C_1 \eps)^2 \langle t \rangle^{\delta},&&|I| \leq N-5.
\endaligned
\right.
\endaligned
\ee
\end{lemma}
\begin{proof}
We apply the energy estimates in Proposition \ref{prop:wave-extra2} with $\gamma=1$ to the $\Gamma^I \widetilde{\Psi}$ equation in \eqref{eq:solution-map2} with $|I| \leq N-1$, we obtain
$$
\aligned
&\big\| \chi^{1/2}(r-t) \langle r-t \rangle \del \Gamma^I \widetilde{\Psi} \big\|
+
\big\| \chi^{1/2}(r-t) \langle r-t \rangle \Gamma^I \widetilde{\Psi} \big\|
\\
\lesssim
&\big\| \langle r \rangle \del \Gamma^I \widetilde{\Psi} (t_0, \cdot) \big\|
+
\big\| \langle r \rangle \Gamma^I \widetilde{\Psi} (t_0, \cdot) \big\|
+
\int_{t_0}^t \big\| \chi^{1/2}(r-\tau) \langle r-\tau \rangle \Gamma^I (\phi \Psi) \big\| \, d\tau.
\endaligned
$$
We note that
$$
\aligned
&\big\| \chi^{1/2}(r-\tau) \langle r-\tau \rangle \Gamma^I (\phi \Psi) \big\|
\\
\lesssim
&\sum_{|I_1|\leq N-8, \, |I_2| \leq N-1} \big\| \chi^{1/2}(r-\tau) \langle r-\tau \rangle \Gamma^{I_1} \phi \big\|_{L^\infty} \big\|  \Gamma^{I_2} \Psi \big\|
\\
+
&\sum_{|I_1|\leq N-1, \, |I_2| \leq N-5} \big\| \chi^{1/2}(r-\tau) \langle r-\tau \rangle \Gamma^{I_1} \phi \big\| \big\| \Gamma^{I_2} \Psi \big\|_{L^\infty}
\\
\lesssim
& (C_1 \eps)^2 \langle \tau \rangle^{-1/2+\delta},
\endaligned
$$
which gives the first inequality in \eqref{eq:est-020}.

Analogously, for the case of $|I| \leq N-5$ we only need to bound
$$
\big\| \chi^{1/2}(r-\tau) \langle r-\tau \rangle \Gamma^I (\phi \Psi) \big\|.
$$
Our strategy is to always take $L^\infty$-norm on the $\Psi$ part, and we find
$$
\aligned
&\big\| \chi^{1/2}(r-\tau) \langle r-\tau \rangle \Gamma^I (\phi \Psi) \big\|
\\
\lesssim
&\sum_{|I_1|\leq N-5, \, |I_2| \leq N-5} \big\| \chi^{1/2}(r-\tau) \langle r-\tau \rangle \Gamma^{I_1} \phi \big\| \big\| \Gamma^{I_2} \Psi \big\|_{L^\infty}
\\
\lesssim
& (C_1 \eps)^2 \langle \tau \rangle^{-1+\delta},
\endaligned
$$
which finishes the proof.

\end{proof}

\begin{lemma}\label{lem:est-020}
The following holds
\bel{eq:est-020}
\aligned
\big\| \chi^{1/2}(r-t) \langle r-t \rangle \Gamma^I \widetilde{\phi} \big\|
&\lesssim
\left\{
\aligned
&\eps+(C_1 \eps)^2 \langle t \rangle^{1/2+\delta},
\qquad
&|I| \leq N-1,
\\
&\eps+(C_1 \eps)^2 \langle t \rangle^{\delta},
\qquad
&|I| \leq N-5.
\endaligned
\right.
\endaligned
\ee

\end{lemma}
\begin{proof}
Our strategy is to first prove the bounds for $\widetilde{\phi}^\Delta$, and then pass them to $\widetilde{\phi}$ according to the relation
$$
\widetilde{\phi} = \Delta \widetilde{\phi}^\Delta.
$$

Consider the $\del \Gamma^I \widetilde{\phi}^\Delta$ equation in \eqref{eq:solution-map2} with $|I| \leq N-1$, and we apply the energy estimates in Proposition \ref{prop:wave-extra2} with $\gamma=1$ to get
$$
\aligned
\big\| \chi^{1/2}(r-t) \langle r-t \rangle \del \del \Gamma^I \widetilde{\phi}^\Delta \big\|
\lesssim
\big\| \langle r \rangle \del \del \Gamma^I \widetilde{\phi}^\Delta (t_0, \cdot) \big\|
+
\int_0^t \big\| \chi^{1/2}(r-\tau) \langle r-\tau \rangle \del \Gamma^I \big| \Psi \big|^2 \big\| \, d\tau.
\endaligned
$$
In succession, we have
$$
\aligned
&\big\| \chi^{1/2}(r-t) \langle r-t \rangle \del \Gamma^I \big| \Psi \big|^2 \big\|
\\
\lesssim
&\sum_{|I_1| \leq N-1, |I_2|\leq N-5} \big\| \chi^{1/2}(r-t) \langle r-t \rangle \del \Gamma^{I_1} \Psi \big\| \big\| \Gamma^{I_2} \Psi \big\|_{L^\infty}
\\
+
&\sum_{|I_1| \leq N-1, |I_2|\leq N-5} \big\| \chi^{1/2}(r-t) \langle r-t \rangle \Gamma^{I_1} \Psi \big\| \big\| \Gamma^{I_2} \Psi \big\|_{L^\infty}
\\
\lesssim
& (C_1 \eps)^2 \langle t \rangle^{-1/2+\delta},
\endaligned
$$
which further gives us
$$
\aligned
&\big\| \chi^{1/2}(r-t) \langle r-t \rangle \del \del \Gamma^I \widetilde{\phi}^\Delta \big\|
\\
\lesssim
&\eps + (C_1 \eps)^2 \int_0^t \langle \tau \rangle^{-1/2+\delta} \, d\tau
\\
\lesssim
&\eps + (C_1 \eps)^2 \langle t \rangle^{1/2+\delta},
\qquad
|I| \leq N-1.
\endaligned
$$
Thus the first inequality in \eqref{eq:est-020} is verified due to
$$
\aligned
&\big\| \chi^{1/2}(r-t) \langle r-t \rangle \Gamma^I \widetilde{\phi} \big\|
\\
=
&\big\| \chi^{1/2}(r-t) \langle r-t \rangle \Gamma^I \Delta \widetilde{\phi}^\Delta \big\|
\\
\lesssim
&\sum_{|I_1| \leq |I|} \big\| \chi^{1/2}(r-t) \langle r-t \rangle \del \del \Gamma^{I_1} \widetilde{\phi}^\Delta \big\|
\\
\lesssim
&\eps + (C_1 \eps)^2 \langle t \rangle^{1/2+\delta},
\qquad
|I| \leq N-1.
\endaligned
$$


The case of $|I|\leq N-5$ can be derived in the same manner. The proof is complete.
\end{proof}

\begin{lemma}\label{lem:est-003}
We get
\bel{eq:est-003}
\big\| \Gamma^I \widetilde{\phi} \big\|
\lesssim
\eps + (C_1 \eps)^2 \langle t \rangle^\delta,
\qquad
|I| \leq N+1
\ee
\end{lemma}
\begin{proof}
As before, our strategy is to first derive the estimates for $\widetilde{\phi}^\Delta$, and then transform the estimates to $\widetilde{\phi}$ via the relation
$$
\widetilde{\phi} = \Delta \widetilde{\phi}^\Delta.
$$

We act $\del \Gamma^I$ with $|I| \leq N+1$ to the $\widetilde{\phi}^\Delta$ equation to get
$$
-\Box \del \Gamma^I \widetilde{\phi}^\Delta = \del \Gamma^I \big(|\Psi|^2\big).
$$
Then the energy estimates \eqref{eq:natural-wKG2} imply
$$
\aligned
\Ecal (t, \del \Gamma^I \widetilde{\phi}^\Delta)^{1/2}
\lesssim
\Ecal (t_0, \del \Gamma^I \widetilde{\phi}^\Delta)^{1/2}
+
\int_0^t \big\| \del \Gamma^I \big(|\Psi|^2\big)  \big\|(\tau) \, d\tau.
\endaligned
$$
Simple analysis shows
$$
\aligned
&\big\| \del \Gamma^I \big(|\Psi|^2\big)  \big\|
\\
\lesssim
&\sum_{\substack{ |I_1| \leq |I| \\ |I_2| \leq N-5 }} \big\| \del \Gamma^{I_1} \Psi \big\| \big\|  \Gamma^{I_2} \Psi \big\|_{L^\infty}
+
\sum_{\substack{|I_1| \leq |I| \\  |I_2| \leq N-6}} \big\| \Gamma^{I_1} \Psi \big\| \big\| \del \Gamma^{I_2} \Psi \big\|_{L^\infty}
\\
\lesssim
&(C_1 \eps)^2 \langle t \rangle^{-1+\delta}.
\endaligned
$$

Thus we have
\bel{eq:est-phi-Delta}
\Ecal (t, \del \Gamma^I \widetilde{\phi}^\Delta)^{1/2}
\lesssim
\eps + (C_1 \eps)^2 \langle t \rangle^\delta
\qquad
|I| \leq N+1.
\ee

Finally, we observe for $|I| \leq N+1$
$$
\big\| \Gamma^I \widetilde{\phi} \big\|
=
\big\| \Gamma^I \Delta \widetilde{\phi}^\Delta \big\|
\lesssim
\sum_{|I_1|\leq N+1} \big\| \del \del \Gamma^{I_1} \widetilde{\phi}^\Delta \big\|
\lesssim 
\eps + (C_1 \eps)^2 \langle t \rangle^{\delta},
$$
which finishes the proof.
\end{proof}

\begin{lemma}\label{lem:est-010}
We have the following estimates
\bel{eq:est-010}
\aligned
\big\| \big( 1-\chi(r-2t) \big)^{1/2} \langle t-r \rangle \Gamma^I \widetilde{\phi} \big\|
&\lesssim
\big( \eps + (C_1 \eps)^2 \big) \langle t \rangle^{\delta},
\qquad
&|I| \leq N.
\endaligned
\ee

\end{lemma}
\begin{proof}
Again, we will first show the bounds for $\widetilde{\phi}^\Delta$, and then pass them to $\widetilde{\phi}$. We only consider the region $\{ r \leq 3t \}$ for large $t$ in the following.

Similar to \eqref{eq:est-phi-Delta}, we have
\bel{eq:est-phi-Delta2}
\Ecal (t, \Gamma^I \widetilde{\phi}^\Delta)^{1/2}
\lesssim
\eps + (C_1 \eps)^2 \langle t \rangle^\delta
\qquad
|I| \leq N+1.
\ee
Recall that we can obtain some extra decay for the Hessian form of the wave components as illustrated in Proposition \ref{prop:wave-extra}, which, for the $\widetilde{\phi}^\Delta$ component in equation \eqref{eq:solution-map1}, reads
$$
\aligned
|\del \del \widetilde{\phi}^\Delta |
\lesssim
{1\over \langle t-r \rangle} \big( |\del \Gamma \widetilde{\phi}^\Delta | + |\del \widetilde{\phi}^\Delta |  \big)
+
{t\over \langle t-r \rangle} |\Psi|^2,
\qquad
r\leq 3t,
\endaligned
$$
in which we used the relation $r \lesssim t$.
To proceed we have
$$
\aligned
&\big( 1-\chi(r-2t) \big)^{1/2}  \langle t-r \rangle |\del \del \widetilde{\phi}^\Delta |
\\
\lesssim
&\big( 1-\chi(r-2t) \big)^{1/2}  \big(|\del \Gamma \widetilde{\phi}^\Delta | + |\del \widetilde{\phi}^\Delta | \big)
+
\big( 1-\chi(r-2t) \big)^{1/2} t |\Psi|^2.
\endaligned
$$
Taking $L^2$-norm and using the simple triangle inequality yield
$$
\aligned
&\big\| \big( 1-\chi(r-2t) \big)^{1/2}  \langle t-r \rangle |\del \del \widetilde{\phi}^\Delta | \big\|
\\
\lesssim
&\big\| \del \Gamma \widetilde{\phi}^\Delta \big\| 
+
\big\| \del \widetilde{\phi}^\Delta \big\|
+
\big\| t |\Psi|^2 \big\|
\\
\lesssim
& \eps + (C_1 \eps)^2 \langle t \rangle^\delta,
\endaligned
$$ 
in which we used \eqref{eq:est-phi-Delta2} in the last step. Thus we obtain
$$
\aligned
&\big\| \big( 1-\chi(r-2t) \big)^{1/2} \langle t-r \rangle  \widetilde{\phi} \big\|
\\
\lesssim
&\big\| \big( 1-\chi(r-2t) \big)^{1/2}  \langle t-r \rangle |\del \del \widetilde{\phi}^\Delta | \big\|
\\
\lesssim
&\eps + (C_1 \eps)^2\langle t \rangle^\delta.
\endaligned
$$

In the same way (with \eqref{eq:est-phi-Delta2}), we get \eqref{eq:est-010}. The proof is done.
\end{proof}

\begin{lemma}\label{lem:est-100}
The following bounds hold true
\bel{eq:est-100}
\aligned
\big\| \big(1-\chi(r/2t) \big)  \langle t-r \rangle^{1-\delta} \Gamma^J \widetilde{\phi}  \big\|
\lesssim
\eps + (C_1 \eps)^2,
\qquad
|J| \leq N-2.
\endaligned
\ee
\end{lemma}
\begin{proof}
We work with the $\Gamma^I \widetilde{\phi}^\Delta$ equation with $|I| \leq N-1$. The energy estimates \eqref{eq:EE-wave} give us
$$
\aligned
&\big\| \big( (t-r)^{-2\delta} \chi(t-r) + 1 - \chi(t-r) \big)^{1/2}   \del \Gamma^I \widetilde{\phi}^\Delta \big\|^2
\\
\lesssim
&\Ecal_{gst} (t_0, \Gamma^I \widetilde{\phi}^\Delta)
+
\int_{t_0}^t \int_{\RR^2} \Big| \big( (\tau-r)^{-2\delta} \chi(\tau-r) + 1 - \chi(\tau-r) \big) \Gamma^I |\Psi|^2 \del_t \Gamma^I \widetilde{\phi}^\Delta \Big| \, dxd\tau.
\endaligned
$$
We need to bound the above spacetime integral, and we find
$$
\aligned
&\int_{t_0}^t \int_{\RR^2} \Big| \big( (\tau-r)^{-2\delta} \chi(\tau-r) + 1 - \chi(\tau-r) \big) \Gamma^I |\Psi|^2 \del_t \Gamma^I \widetilde{\phi}^\Delta \Big| \, dxd\tau
\\
\lesssim
& \int_{t_0}^t \Big\| \big( (\tau-r)^{-2\delta} \chi(\tau-r) + 1 - \chi(\tau-r) \big) \Gamma^I |\Psi|^2 \Big\| \big\| \del_t \Gamma^I \widetilde{\phi}^\Delta \big\| \, d\tau
\\
\lesssim
& C_1 \eps \int_{t_0}^t \Big\| \big( (\tau-r)^{-2\delta} \chi(\tau-r) + 1 - \chi(\tau-r) \big) \Gamma^I |\Psi|^2 \Big\| \langle \tau \rangle^{\delta} \, d\tau.
\endaligned
$$
We then do the estimates in different regions (note the relation $1\lesssim \la t-r \ra \lesssim 1$ holds when $|t-r| \lesssim 1$), and we proceed to have
$$
\aligned
&\int_{t_0}^t \Big\| \big( (\tau-r)^{-2\delta} \chi(\tau-r) + 1 - \chi(\tau-r) \big) \Gamma^I |\Psi|^2 \Big\| \langle \tau \rangle^{\delta} \, d\tau
\\
\lesssim
&\int_{t_0}^t \big\|  \langle \tau-r\rangle^{-2\delta}  \Gamma^I |\Psi|^2 \big\| \langle \tau \rangle^{\delta} \, d\tau
+
\int_{t_0}^t \big\|  \chi(r-\tau)  \Gamma^I |\Psi|^2 \big\| \langle \tau \rangle^{\delta} \, d\tau
\\
=: 
& A_1 +A_2.
\endaligned
$$
To estimate $A_1$ we utilise the spacetime integral bounds in the ghost weight energy estimates to get
$$
\aligned
A_1
\lesssim
&\sum_{|I_1| \leq N-1, |I_2| \leq N-7} \int_{t_0}^t \big\| \Gamma^{I_1} \Psi \big\| 
\big\| \langle \tau-r\rangle^{-2\delta} \Gamma^{I_2} \Psi \big\|_{L^\infty} \langle \tau \rangle^{\delta}  \, d\tau
\\
\lesssim
&(C_1 \eps)^2  \int_{t_0}^t \langle \tau \rangle^{-1-\delta} \, d\tau
\\
\lesssim
&(C_1 \eps)^2.
\endaligned
$$
For the term $A_2$, we have
$$
\aligned
A_2
\lesssim
&\sum_{|I_1|\leq N-1, |I_2| \leq N-8} \int_{t_0}^t \big\|  \Gamma^{I_1} \Psi \big\| \big\|  \chi(r-\tau)  \Gamma^{I_2} \Psi| \big\|_{L^\infty} \langle \tau \rangle^{\delta} \, d\tau
\\
\lesssim
&(C_1 \eps)^2 \int_{t_0}^t \langle \tau \rangle^{-5/4+3\delta} \, d\tau
\lesssim
(C_1 \eps)^2.
\endaligned
$$

Gathering the above estimates, we arrive at
$$
\big\| \big( (t-r)^{-2\delta} \chi(t-r) + 1 - \chi(t-r) \big)^{1/2}   \del \Gamma^I \widetilde{\phi}^\Delta \big\|
\lesssim
\eps + (C_1 \eps)^{3/2},
\qquad
|I| \leq N-1.
$$

Finally, recall again the estimates for the Hessian of wave component in Proposition \ref{prop:wave-extra}
$$
\big| \del \del \Gamma^J \widetilde{\phi}^\Delta \big|
\lesssim
{1\over \langle t-r\rangle} \sum_{|I| \leq |J|+1} \big| \del  \Gamma^J \widetilde{\phi}^\Delta\big|  + {\langle t \rangle \over \langle t-r \rangle} \big|\Gamma^J |\Psi|^2 \big|,
\qquad
r\leq 3t,
$$
and we further obtain (for $|J|\leq N-2$)
$$
\aligned
&\big\| \big(1-\chi(r/2t) \big)  \langle t-r \rangle^{1-\delta} \Gamma^J \widetilde{\phi}  \big\|
\\
\lesssim
&\big\| \big( (t-r)^{-2\delta} \chi(t-r) + 1 - \chi(t-r) \big)^{1/2} \langle t-r \rangle \Gamma^J \widetilde{\phi}  \big\|
\\
\lesssim
&\sum_{|J_1| \leq N-2} \big\| \big( (t-r)^{-2\delta} \chi(t-r) + 1 - \chi(t-r) \big)^{1/2} \langle t-r \rangle \del \del \Gamma^{J_1}  \widetilde{\phi}^\Delta  \big\|
\\
\lesssim
&\eps + (C_1 \eps)^2.
\endaligned
$$

The proof is done.
\end{proof}

\begin{lemma}\label{lem:est-400}
We have for $|I|\leq N-5$
\bel{eq:200}
\big| \Gamma^I \widetilde{\phi} \big|
\lesssim
\big( \eps + (C_1 \eps)^2 \big) \langle t-r \rangle^{-1+\delta} \langle t+r \rangle^{-1/2},
\qquad
t/2 \leq r \leq 2t.
\ee
\end{lemma}
\begin{proof}
The proof follows from Lemma \ref{lem:est-100} and the weighted Sobolev inequality in Proposition \ref{prop:S}.

For $|I|\leq N-5$ we have (for $\Lambda \in \{ \del_r, \Omega_{ab} \}$)
$$
\aligned
&\langle r\rangle^{1/2} \big| \big(1-\chi(r/2t) \big) \chi(1/2+2r/t) \langle t-r \rangle^{1-\delta}  \Gamma^I \widetilde{\phi} \big|
\\
\lesssim
&\sum_{|J|\leq 3} \sup_{t\geq t_0} \Big\| \Lambda^J \Big( \big(1-\chi(r/2t) \big) \chi(1/2+2r/t) \langle t-r \rangle^{1-\delta} \Gamma^I \widetilde{\phi} \Big) \Big\|
\endaligned
$$

Note that the rotation vector field $\Omega$ commutes with $r, t$ (see also Proposition \ref{prop:S}), which gives us
$$
\aligned
&\sum_{|J|\leq 3} \Big\| \Lambda^J \Big( \big(1-\chi(r/2t) \big) \chi(1/2+2r/t) \langle t-r \rangle^{1-\delta} \Gamma^I \widetilde{\phi} \Big) \Big\|
\\
\lesssim
&\sum_{|I_1|\leq N-2} \Big\|  \big(1-\chi(r/2t) \big) \chi(1/2+2r/t) \langle t-r \rangle^{1-\delta} \Gamma^{I_1} \widetilde{\phi}  \Big\|
\\
+
&\sum_{|I_1|\leq N-3}\Big\| \la  t\ra^{-1}  \langle t-r \rangle^{1-\delta} \Gamma^{I_1} \widetilde{\phi}  \Big\|
\\
\lesssim
&\eps + (C_1 \eps)^2,
\endaligned
$$
which leads us to
$$
\aligned
&\big| \big(1-\chi(r/2t) \big) \chi(1/2+2r/t) \langle t-r \rangle^{1-\delta}  \Gamma^I \widetilde{\phi} \big|
\\
\lesssim
&\big( \eps + (C_1 \eps)^2 \big) \langle r\rangle^{-1/2},
\qquad
|I| \leq N-5.
\endaligned
$$

The proof is complete by noting $\la t+r\ra \lesssim \la r\ra \lesssim \la t+r\ra$ when $r \geq t/2$.
\end{proof}

%
%

\begin{lemma}
The following pointwise bounds are valid
\bel{eq:est-050}
\big| \Gamma^I \widetilde{\phi} \big|
\lesssim
\big( \eps + (C_1 \eps)^2 \big) \langle t-r \rangle^{-1+\delta} \langle t+r \rangle^{-1/2},
\qquad
|I| \leq N-8.
\ee
\end{lemma}
\begin{proof}
Thanks to the estimates in Lemma \ref{lem:est-400}, we only need to show \eqref{eq:est-050} holds in the regions $\{ r \leq t/2 \}$ and $\{ r \geq 2t\}$ for large $t$.

\paragraph{Case I: $\{ r \leq t/2 \}$.}
Consider first the $\Gamma^{J_1} \widetilde{\phi}^\Delta$ equations with $|J_1| \leq N-1$, and the energy estimates \eqref{eq:gst-wKG2} yields
$$
\Ecal_{gst} (t, \Gamma^{J_1} \widetilde{\phi}^\Delta)^{1/2}
\lesssim
\Ecal_{gst} (t_0, \Gamma^{J_1} \widetilde{\phi}^\Delta)^{1/2}
+
\int_{t_0}^t \big\| \Gamma^{J_1} |\Psi|^2 \big\| \, d\tau.
$$
We proceed to bound
$$
\aligned
\big\| \Gamma^{J_1} |\Psi|^2 \big\|
\lesssim
\sum_{|I_1| \leq N-1, \, |I_2| \leq N-5} \big\| \Gamma^{I_1} \Psi \big\| \, \big\| \Gamma^{I_2} \Psi \big\|_{L^\infty}
\lesssim
(C_1 \eps)^2 \langle \tau \rangle^{-1}.
\endaligned
$$
Thus we have
$$
\Ecal_{gst} (t, \Gamma^{J_1} \widetilde{\phi}^\Delta)^{1/2}
\lesssim
\epsilon
+
(C_1 \eps)^2 \int_{t_0}^t \langle \tau \rangle^{-1} \, d\tau
\lesssim
\eps + (C_1 \eps)^2 \langle t \rangle^{\delta},
\qquad
|J_1| \leq N-1.
$$

Next, we apply the Klainerman-Sobolev inequality in Proposition \ref{prop:K-S} to get
$$
\big| \del \Gamma^J \widetilde{\phi}^\Delta \big|
\lesssim
\big( \eps + (C_1 \eps)^2 \big) \langle t+r \rangle^{-1/2+\delta/2},
\qquad
|J| \leq N-4.
$$
Then Proposition \ref{prop:wave-extra} allows us to obtain extra $\langle t-r \rangle$-decay with one more derivative, i.e.,
$$
\big| \del \del \Gamma^J \widetilde{\phi}^\Delta \big|
\lesssim
\big( \eps + (C_1 \eps)^2 \big) \langle t-r \rangle^{-1} \langle t+r \rangle^{-1/2+\delta/2},
\qquad
|J| \leq N-5.
$$
Finally, recalling the relation $\widetilde{\phi} = \Delta \widetilde{\phi}^\Delta$ gives us
$$
\big| \Gamma^J \widetilde{\phi} \big|
=
\big| \Gamma^J \Delta \widetilde{\phi}^\Delta \big|
\lesssim
\sum_{|I_1| \leq |J|} \big| \del \del \Gamma^I \widetilde{\phi}^\Delta \big|
\lesssim
\big(\eps + (C_1 \eps)^2\big)  \langle t-r \rangle^{-1} \langle t+r \rangle^{-1/2+\delta},
\qquad
|J| \leq N-5,
$$
and hence for $|J| \leq N-5$ we have
$$
\big| \Gamma^J \widetilde{\phi} \big|
\lesssim
\big(\eps + (C_1 \eps)^2\big) \langle t-r \rangle^{-1+\delta} \langle t+r \rangle^{-1/2},
\qquad
\{r \leq t/2\}.
$$

\paragraph{Case II: $\{ r \geq 2t \}$.}
Recall the estimates in Lemma \ref{lem:est-020}
$$
\big\| \chi^{1/2}(r-t) \langle r-t \rangle \Gamma^J \widetilde{\phi} \big\|
\lesssim
\eps+(C_1 \eps)^2 \langle t \rangle^{\delta},
\qquad
|J| \leq N-5,
$$
and this deduces that for large $t$ it holds   
$$
\big\| \chi(-6t/r + 5) \langle r-t \rangle \Gamma^J \widetilde{\phi} \big\|
\lesssim
\eps+(C_1 \eps)^2 \langle t \rangle^{\delta},
\qquad
|J| \leq N-5,
$$
in which $\chi(-6t/r + 5) $ is $1$ for $r\geq 2t$, and $0$ for $r\leq 3t/2$.
Then we apply again the weighted Sobolev inequality in Proposition \ref{prop:S} to derive
$$
\chi(-6t/r + 5) \langle r-t \rangle \big| \Gamma^I \widetilde{\phi} \big|
\lesssim
\big(  \eps+(C_1 \eps)^2 \big) \langle r \rangle^{-1/2+\delta},
\qquad
|I| \leq N-8.
$$

Finally, combining the afore obtained results in Lemma \ref{lem:est-400}, we finish the proof.
\end{proof}

\begin{lemma}
We have the following uniform bounds
\bel{eq:est-030}
\aligned
\Ecal_{gst, 1} (t, \Gamma^I \widetilde{\Psi})^{1/2}
&\lesssim
\eps + (C_1 \eps)^{3/2},
\qquad
&|I| \leq N-1,
\\
\Ecal_{gst} (t, \Gamma^I \widetilde{\phi})^{1/2}
&\lesssim
\eps + (C_1 \eps)^{3/2},
\qquad
&|I| \leq N-2.
\endaligned
\ee
\end{lemma}
\begin{proof}
Our strategy is to divide the spacetime region roughly into two parts $\{r\geq 2t \}, \{r\leq 2t\}$, and then conduct the estimates in different parts.

For the Klein-Gordon part $\widetilde{\Psi}$, the energy estimates \eqref{eq:gst-wKG2} give us for $|I| \leq N-1$
$$
\aligned
\Ecal_{gst, 1} (t, \Gamma^I \widetilde{\Psi})^{1/2}
\lesssim
\Ecal_{gst, 1} (t_0, \Gamma^I \widetilde{\Psi})^{1/2}
+
\int_{t_0}^t \big\| \Gamma^I (\phi \, \Psi) \big\| \, d\tau.
\endaligned
$$
We have
$$
\aligned
\int_{t_0}^t \big\| \Gamma^I (\phi \, \Psi) \big\| \, d\tau
\lesssim
&\int_{t_0}^t \big\| \chi(r-2\tau) \Gamma^I (\phi \, \Psi) \big\| \, d\tau
+
\int_{t_0}^t \big\| \big( 1-\chi(r-2\tau) \big) \Gamma^I (\phi \, \Psi) \big\| \, d\tau
\\
:=
& A_1 + A_2.
\endaligned
$$
We next bound these two terms separately. On one hand, we find
$$
\aligned
A_1
\lesssim
&\sum_{|I_1|+|I_2|\leq |I|} \int_{t_0}^t \big\| \chi(r-2\tau) \Gamma^{I_1} \phi \, \Gamma^{I_2} \Psi) \big\| \, d\tau
\\
\lesssim
&\sum_{|I_1|\leq |I|, |I_2|\leq N-5} \int_{t_0}^t \big\| \chi(r-2\tau) \langle r-\tau \rangle \Gamma^{I_1} \phi \big\| \, \big\|\langle \tau 
\rangle^{-1}\Gamma^{I_2} \Psi \big\|_{L^\infty} \, d\tau
\\
+
&\sum_{|I_1|\leq N-8, |I_2|\leq |I|} \int_{t_0}^t \big\| \chi(r-2\tau) \Gamma^{I_1} \phi \big\|_{L^\infty} \, \big\| \Gamma^{I_2} \Psi \big\| \, d\tau
\\
\lesssim
&(C_1 \eps)^2 \int_{t_0}^t \langle \tau \rangle^{-3/2+2\delta} \, d\tau
\lesssim (C_1 \eps)^2.
\endaligned
$$
On the other hand, we have
$$
\aligned
A_2
\lesssim
&\sum_{|I_1|+|I_2|\leq |I|} \int_{t_0}^t \big\| \big( 1-\chi(r-2\tau) \big) \Gamma^{I_1} \phi \, \Gamma^{I_2} \Psi \big\| \, d\tau
\\
\lesssim
&\sum_{\substack{|I_1|\leq |I|\\ |I_2|\leq N-7}} \int_{t_0}^t \big\| \big( 1-\chi(r-2\tau) \big)^{1/2} \langle r-\tau \rangle \Gamma^{I_1} \phi \big\| \, \big\|\big( 1-\chi(r-2\tau) \big)^{1/2}\langle r-\tau \rangle^{-1}\Gamma^{I_2} \Psi \big\|_{L^\infty} \, d\tau
\\
+
&\sum_{\substack{|I_1|\leq N-8\\ |I_2|\leq |I|}} \int_{t_0}^t \big\| \big( 1-\chi(r-2\tau) \Big)^{1/2} {\langle r-\tau \rangle \over \langle r+\tau \rangle} \Gamma^{I_1} \phi \Big\|_{L^\infty} \, \Big\| \big( 1-\chi(r-2\tau) \big)^{1/2} {\langle t+\tau \rangle \over \langle r-\tau \rangle} \Gamma^{I_2} \Psi \Big\| \, d\tau
\\
\lesssim
&(C_1 \eps)^2 \int_{t_0}^t \langle \tau \rangle^{-3/2+2\delta} \, d\tau
\lesssim (C_1 \eps)^2.
\endaligned
$$

Thus we are led to
$$
\Ecal_{gst, 1} (t, \Gamma^I \widetilde{\Psi})^{1/2}
\lesssim
\eps + (C_1 \eps)^2,
\qquad
|I| \leq N-1.
$$

For the wave part $\widetilde{\phi}$, we have, according to the energy estimates \eqref{eq:gst-wKG1}, that
$$
\aligned
\Ecal_{gst} (t, \Gamma^I \widetilde{\phi})
\lesssim
\Ecal_{gst} (t_0, \Gamma^I \widetilde{\phi})
+
\int_{t_0}^t \int_{\RR^2} \big| \Gamma^I \Delta |\Psi|^2 \, \del_t \Gamma^I \widetilde{\phi} \big| \, dx d\tau.
\endaligned
$$
Recall the estimates in Lemmas \ref{lem:est-020} and \ref{lem:est-010}, for $|I| \leq N-2$ we have 
$$
\aligned
&\big\| \langle \tau-r \rangle \del_t \Gamma^I \widetilde{\phi} \big\|
\\
\lesssim
&\big\| \chi(r-2\tau) \langle \tau-r \rangle \del_t \Gamma^I \widetilde{\phi} \big\|
+
\big\| \big( 1-\chi(r-2\tau) \big) \langle \tau-r \rangle \del_t \Gamma^I \widetilde{\phi} \big\|
\lesssim
C_1 \eps \langle \tau \rangle^{1/2+\delta}.
\endaligned
$$
Successively, we obtain
$$
\aligned
&\int_{t_0}^t \int_{\RR^2} \big| \Gamma^I \Delta |\Psi|^2 \, \del_t \Gamma^I \widetilde{\phi} \big| \, dx d\tau
\\
\lesssim
&\int_{t_0}^t \big\| \langle \tau - r \rangle^{-1} \Gamma^I \Delta |\Psi|^2 \big\|   
\big\| \langle \tau-r \rangle \del_t \Gamma^I \widetilde{\phi} \big\| \, d\tau
\\
\lesssim
&C_1 \eps \int_{t_0}^t \langle \tau \rangle^{1/2+\delta} \big\| \langle \tau - r \rangle^{-1} \Gamma^I \Delta |\Psi|^2 \big\| \, d\tau.
\endaligned
$$
By the pointwise decay for the Klein-Gordon field $\Psi$, we get
$$
\sum_{|I_1|\leq N-7}\big\| \langle \tau-r \rangle^{-1} \Gamma^{I_1} \Psi \big\|_{L^\infty}
\lesssim
C_1 \eps \langle \tau \rangle^{-2},
$$
which further gives us
$$
\big\| \langle \tau - r \rangle^{-1} \Gamma^I \Delta |\Psi|^2 \big\|
\lesssim
\sum_{\substack{|I_1|\leq N-7 \\ |I_2| \leq N}} \big\| \langle \tau-r \rangle^{-1} \Gamma^{I_1} \Psi \big\|_{L^\infty} \big\| \Gamma^{I_2} \Psi \big\|
\lesssim
(C_1 \eps)^2 \langle \tau \rangle^{-2+\delta}.
$$
Thus we arrive at
$$
\Ecal_{gst} (t, \Gamma^I \widetilde{\phi})
\lesssim
\eps^2
+
(C_1 \eps)^3 \int_{t_0}^t \langle \tau \rangle^{-3/2+2\delta} \, d\tau
\lesssim
\eps^2 + (C_1 \eps)^3,
\qquad
|I| \leq N-2.
$$

The proof is done.
\end{proof}

\begin{lemma}\label{lem:est-800}
We get
\bel{eq:est-500}
\chi(r-t)  \big|\Gamma^I \widetilde{\Psi} \big|
\lesssim
\big( \eps + (C_1 \eps)^{3/2} \big) \langle t+r\rangle^{-5/4+\delta},
\qquad
|I|\leq N-8.
\ee
\end{lemma}
\begin{proof}
We will rely on the weighted energy estimates in Lemma \ref{lem:est-020} and the weighted Sobolev inequality in Proposition \ref{prop:S} to derive the pointwise estimates in \eqref{eq:est-500}.

Applying the weighted Sobolev inequality in Proposition \ref{prop:S} implies
$$
\chi(r-t) \langle t-r \rangle \langle r\rangle^{1/2} \big|\Gamma^I \widetilde{\Psi} \big|
\lesssim
\sum_{ |J| \leq 3} \big\| \Omega^J \big( \chi(r-t) \langle t-r \rangle \Gamma^I \widetilde{\Psi} \big) \big\|.
$$
We note within the support of $\chi'(r-t)$ it holds that $1\lesssim \langle t-r \rangle \lesssim 1$, and by the commutator estimates $\Omega_{ab} r = \Omega_{ab} t = 0$ we find
$$
\sum_{ |J| \leq 3} \big\| \Omega^J \big( \chi(r-t) \langle t-r \rangle \Gamma^I \widetilde{\Psi} \big) \big\|
\lesssim
\sum_{ |I_1| \leq N-5} \big\|  \chi(r-t) \langle t-r \rangle \Gamma^{I_1} \widetilde{\Psi} \big\|.
$$
Then by Lemma \ref{lem:est-020}, we get
$$
\sum_{ |J| \leq 3} \big\| \Omega^J \big( \chi(r-t) \langle t-r \rangle \Gamma^I \widetilde{\Psi} \big) \big\|
\lesssim
\eps + (C_1 \eps)^2 \langle t\rangle^\delta,
$$
and hence
$$
\chi(r-t)  \big|\Gamma^I \widetilde{\Psi} \big|
\lesssim
\big(\eps + (C_1 \eps)^2  \big) \langle t-r \rangle^{-1} \langle r\rangle^{-1/2+\delta},
\qquad
|I| \leq N-8.
$$

Finally, by the aid of Lemma \ref{lem:est-005c}, we are led to
$$
\aligned
&\chi(r-t)  \big|\Gamma^I \widetilde{\Psi} \big|
\\
\lesssim
&\chi^{1/2}(r-t)  \big|\Gamma^I \widetilde{\Psi} \big|^{1/2} \big|\Gamma^I \widetilde{\Psi} \big|^{1/2}
\\
\lesssim
&\big( \eps+ (C_1 \eps)^2 \big) \langle t-r \rangle^{-1/2} \langle t+r\rangle^{-1/4+\delta/2} \langle t-r\rangle^{1/2} \langle t+r\rangle^{-1}
\\
\lesssim
&\big( \eps+ (C_1 \eps)^2 \big) \langle t+r\rangle^{-5/4+\delta},
\qquad
|I| \leq N-8.
\endaligned
$$

The proof is done.
\end{proof}

%
%
%

%
%
%

%
%
%

\begin{proof}[Proof of Proposition \ref{prop:mapping1}]
Gathering the results obtained in Lemmas \ref{lem:est-001}--\ref{lem:est-800}, we get
$$
\big\| \big(\widetilde{\Psi}, \widetilde{\phi} \big) \big\|_X
\leq C \eps + C (C_1 \eps)^2.
$$
By choosing large $C_1 \gg 1$, and sufficiently small $\eps \ll 1$, such that $C_1 \eps \ll 1$, we arrive at
$$
\big\| \big(\widetilde{\Psi}, \widetilde{\phi} \big) \big\|_X
\leq {1\over 2} C_1 \eps,
$$
which means 
$$
\big(\widetilde{\Psi}, \widetilde{\phi} \big) \in X.
$$

In almost the same way, we can show \eqref{eq:contraction2} (we might further shrink the size of $\eps$). Hence the proof is complete.
\end{proof}

\begin{proof}[Proof of Theorem \ref{thm:main1}]
The Banach fixed point theorem together with Proposition \ref{prop:mapping1} leads us to Theorem \ref{thm:main1}.
\end{proof}


\section{Scattering}\label{sec:scatter}

In this Section, we briefly discuss about the scattering of the Klein-Gordon-Zakharov system \eqref{eq:model-KGZ} in $\RR^{1+2}$. We show that the Klein-Gordon field $E$ scatters to a linear Klein-Gordon equation in its high-order energy space (i.e., $\|E \|_{H^{N-7}} + \|\del_t E\|_{H^{N-8}}$), but whether the wave field $n$ scatters (linearly or nonlinearly) is unknown. We also note that this is different from the scattering result obtained in \cite{Guo-N-W} for Klein-Gordon-Zakharov equations in $\RR^{1+3}$, where the initial data are assumed to lie in the low regularity space and different difficulties arise.

We need one key fundamental result from \cite{Katayama17} (Lemma 6.12 there), which originally provides a sufficient condition for the linear scattering of wave equations, but it extends to Klein-Gordon cases with similar proof. We now give the statement of the fundamental result and its proof can be found in either \cite{Katayama17} or Appendix \ref{sec:appendix}.

\begin{lemma}\label{lem:scatter}
Consider the Klein-Gordon equation 
$$
\aligned
&-\Box u +  u = Q(t, x),
\\
&\big( u, \del_t u \big)(t_0=0) = (u_0, u_1).
\endaligned
$$
If it holds (with $N_1 \geq 1$ an integer)
\bel{eq:scatter-condition}
\int_0^{+\infty} \| Q(\tau, \cdot) \|_{H^{N_1}} \, d\tau
< +\infty,
\ee
then there exist $(u_0^+, u_1^+) \in H^{N_1+1} \times H^{N_1}$ and a free Klein-Gordon component $u^+$ satisfying
$$
\aligned
&-\Box u^+ + u^+ = 0,
\\
&\big( u^+, \del_t u^+ \big)(t_0=0) = (u_0^+, u_1^+),
\endaligned
$$
such that $u$ scatters to $u^+$, i.e.,
\be 
\| (u-u^+) (t) \|_{H^{N_1+1}} + m \|\del_t ( u-u^+)(t) \|_{H^{N_1}}
\leq
C \int_t^{+\infty} \| Q(\tau, \cdot) \| \, d\tau
\to 0,
\qquad
\text{as } t\to +\infty.
\ee
\end{lemma}

\begin{remark}
By Lemma \ref{lem:scatter} and the results obtained in \cite{Dong2101}, we know that the Klein-Gordon-Zakharov equations enjoy linear scattering in $\RR^{1+3}$, which was also shown in \cite{OTT} with high regular initial data. But again we want to emphasize that there are different difficulties arising in obtaining scattering results for data lying in low regularity space as studied in \cite{Guo-N-W} on Klein-Gordon-Zakharov equations.
\end{remark}

\begin{proof}[Proof of Theorem \ref{thm:scatter}]

According to Lemma \ref{lem:scatter}, we only need to verify 
\bel{eq:scatter-proof}
\big\| n E \big\|_{H^{N-8}} 
\lesssim (C_1 \eps)^2 \langle t \rangle^{-5/4}.
\ee
By the definition of the $\| \cdot \|$-norm in \eqref{eq:X-norm}, we find
$$
\aligned
&\big\| n E \big\|_{H^{N-8}} 
\\
\lesssim
&\sum_{|J|\leq N-8} \big\| \Gamma^J (n E) \big\|
\\
\lesssim
&\sum_{|J_1|, |J_2|\leq N-8}  \Big\| {\langle t-r \rangle \over \langle t+r \rangle} \Gamma^{J_1} n \Big\|_{L^\infty}  \Big\| {\langle t+r \rangle \over \langle t-r \rangle} \Gamma^{J_2} E \Big\|
\\
\lesssim
& (C_1 \eps)^2 \langle t \rangle^{-3/2+2\delta}.
\endaligned
$$
By the smallness of $\delta$, we thus arrive at \eqref{eq:scatter-proof}, and hence Theorem \ref{thm:scatter}.

\end{proof}

\appendix
\section{Proof of Lemma \ref{lem:scatter}}\label{sec:appendix}
By the linear theory on wave equations, the free linear Klein-Gordon equation generates a strongly continuous semi-group acting on $H^{N_1+1}\times H^{N_1}$ as follows (with $N_1 \geq 1$ an integer). Let $(u_0,u_1)\in H^{N_1+1}\times H^{N_1}$. Then the Cauchy problem
$$
-\Box u + u = 0,\quad u(0,x) = u_0(x),\quad \del_t u(0,x) = u_1(x)
$$
generates a unique global solution 
$(u(x),\del_tu(x))\in C([0,\infty),H^{N_1+1}\times H^{N_1})\cap C^1([0,\infty),H^{N_1}\times H^{N_1-1})$ (a detailed proof can be found in \cite{Sogge}). This leads to 
$$
\Scal_1(t): (u_0,u_1)\mapsto \big(u(t,\cdot), \del_tu(t,\cdot)\big)\in H^{N_1+1}\times H^{N_1},
$$
with ($T$ below means the transpose of a matrix)
$$
\Scal_1(t) (u_0,u_1) 
=
\Big(e^{t\mathcal{A}_1} 
\begin{pmatrix}
u_0
\\
u_1
\end{pmatrix}
\Big)^T,
\qquad
\mathcal{A}_1
=
\begin{pmatrix}
0& 1
\\
\Delta-1& 0
\end{pmatrix}.
$$
By energy identity, $\Scal_1(t)$ is unitary for all $t\geq 0$. By the invariance under time translation and global uniqueness,
$$
\Scal_1(t+s) = \Scal_1(t)\circ\Scal_1(s).
$$
By the fact that  $(u(x),\del_tu(x))\in C([0,\infty),H^{N_1+1}\times H^{N_1})$,   
$$
\Scal_1(t):(u_0,u_1)\mapsto (u,\del_t u)\quad\text{in}\quad H^{N_1+1}\times H^{N_1}.
$$
That is, $\Scal_1(t)$ is a strongly continuous semi-group.  Next, we consider a non-homogeneous case.
\be 
\aligned
&-\Box u +  u = Q(t, x),
\\
&\big( u, \del_t u \big)(t_0=0) = (u_0, u_1).
\endaligned
\ee
where $Q(t,\cdot)$ is supposed to be in $L^1([0,\infty),H^{N_1})$. This equation has a unique global solution in $ C([0,\infty),H^{N_1+1}\times H^{N_1})\cap C^1([0,\infty),H^{N_1}\times H^{N_1-1})$ (see also a detailed proof in \cite{Sogge}).
By Duhamel's principle (which is guaranteed by the strong continuity of $\Scal_1(t)$), the associated global solution  can be written as 
$$
(u,\del_tu)(t) = \Scal_1(t)(u_0,u_1) + \int_0^t\Scal_1(t-\tau)(0,Q(\tau))d\tau.
$$
Inspired by this formula, we set the initial data $(u_0^+, u_1^+)$ to be
\be 
(u_0^+, u_1^+)
=
(u_0, u_1)
+
\int_0^{+\infty} \mathcal{S}_1(-\tau) (0, Q(\tau, x)) \, d\tau,
\ee
which is well-defined in $H^{N_1+1}\times H^{N_1}$ as long as
$$
\int_0^{+\infty} \| Q(\tau, x) \|_{H^{N_1}} \, d\tau
<+\infty.
$$
Then we observe that
\be 
\aligned
&\big\| (u,\del_t u) - \mathcal{S}_1(t) (u_0^+, u_1^+) \big\|_{H^{N_1+1}\times H^{N_1}}
\\
=
&\Big\| \int_t^{+\infty} \mathcal{S}_1(t-\tau) (0, Q(\tau, x)) \, d\tau \Big\|_{H^{N_1+1}\times H^{N_1}}
\\
\lesssim
&\int_t^{+\infty} \big\| Q(\tau, x) \big\|_{H^{N_1}} \, d\tau,
\endaligned
\ee
which finishes the proof of Lemma \ref{lem:scatter}.


\section*{Acknowledgements} 

The authors are grateful to Prof. Zihua Guo (Monash University) for leading them to study the scattering aspect of the Klein-Gordon-Zakharov equations. The authors would also like to thank Dr. Kuijie Li (Nankai University) and Dr. Zoe Wyatt (Cambridge University), for many helpful discussions. 





{\footnotesize

\end{document}